\documentclass[12pt]{article}

\usepackage{amssymb}
\usepackage{amsthm}
\usepackage{amsmath}
\usepackage{graphicx}
\usepackage{fullpage}
\usepackage{color}
 \numberwithin{equation}{section}

\theoremstyle{plain}
\newtheorem{thm}{Theorem}[section]
\newtheorem{cor}[thm]{Corollary}

\newtheorem{lem}[thm]{Lemma}
\newtheorem{prop}[thm]{Proposition}

\newtheorem{conj}[thm]{Conjecture}

\theoremstyle{definition}
\newtheorem{defn}[thm]{Definition}

\theoremstyle{remark}

\newtheorem{rem}[thm]{Remark}

\newcommand{\N}{\mathbb{N}}
\newcommand{\R}{\mathbb{R}}

\newcommand{\E}{\mathbb{E}}

\newcommand{\I}{\infty}
\newcommand{\tr}{\text{tr}}
\newcommand{\bp}{\begin{proof}[\ensuremath{\mathbf{Proof}}]}
\newcommand{\bs}{\begin{proof}[\ensuremath{\mathbf{Solution}}]}
\newcommand{\ep}{\end{proof}}

\begin{document}

\title{The eigenvalue problem of singular ergodic control}

\author{Ryan Hynd\thanks{This material is based upon work supported by the National Science Foundation under Grant No. DMS-1004733.}\\
Courant Institute of Mathematical Sciences\\
New York University\\
251 Mercer Street\\
New York, NY 10012-1185 USA} 

\maketitle

\begin{abstract}
We consider the problem of finding a real number $\lambda$ and a function $u$ satisfying the PDE
$$
\max\{\lambda -\Delta u -f,|Du|-1 \}=0, \quad x\in \R^n.
$$
Here $f$ is a convex, superlinear function.  We prove that there is a unique $\lambda^*$ such that the above PDE has a viscosity
solution $u$ satisfying $\lim_{|x|\rightarrow \infty}u(x)/|x|=1$. Moreover, we show that associated to $\lambda^*$ is a convex solution $u^*$ with 
$D^2u^*\in L^\infty(\R^n)$ and give a min-max formula for $\lambda^*$.  $\lambda^*$ has a probabilistic interpretation as being the least, long-time averaged (``ergodic") cost for a singular
control problem involving $f$. 
\end{abstract}

\tableofcontents

\section{Introduction}
In this paper, we address the following problem: find $\lambda\in \R$ such that the PDE
\begin{equation}\label{lamPDE}
\max\left\{\lambda  - \Delta u  - f, |Du| -1 \right\}=0, \quad x\in\R^n
\end{equation}
has a solution $u:\R^n\rightarrow \R$.  We call any such $\lambda$ an {\it eigenvalue}.  
Here is it assumed that $f\in C^\I(\R^n)$ is convex and satisfies the growth condition
\begin{equation}\label{hgrowth}
\lim_{|x|\rightarrow \infty}\frac{f(x)}{|x|}=+\infty. 
\end{equation}
These assumptions imply that $f$ is bounded from below and without any loss of generality it will also be assumed that $f$ is non-negative.

\par Our main result is

\begin{thm}\label{mainthm}
There is a unique $\lambda^*\in\R$ such that \eqref{lamPDE} has a viscosity solution $u\in C(\R^n)$ satisfying 
\begin{equation}\label{ugrowth}
\lim_{|x|\rightarrow \infty}\frac{u(x)}{|x|}=1. 
\end{equation}
Moreover, associated to this eigenvalue $\lambda^*\in\R$ is a convex solution $u^*$ of \eqref{lamPDE} satisfying \eqref{ugrowth}
such that $D^{2}u^*\in L^\I(\R^n)$. 
\end{thm}

\par  Our approach to proving this theorem is as follows.  We use an idea inspired from the maximum principle for solutions of second order elliptic PDE to establish a comparison principle among eigenvalues.  Then we present a simple technique for approximating the value of an eigenvalue. It is based on studying the related elliptic PDE

\begin{equation}\label{deltaPDE}
\max\left\{\delta u  - \Delta u  - f, |Du| -1 \right\}=0, \quad x\in\R^n
\end{equation}
for $\delta$ small and positive. In particular, we prove the following

\begin{thm}\label{deltathm}
(i) For each $\delta>0$, there is a unique viscosity solution $u_\delta$ of \eqref{deltaPDE} satisfying \eqref{ugrowth}.  \\
(ii) There is a universal constant $L$ such that
$$
0\le D^2u_\delta(x)\le L, \quad \text{a.e.}\; x\in \R^n
$$
for $0<\delta\le 1$. \\
(iii) For every $x_0\in \R^n$, there is a sequence $(\delta_k)_{k\in\N}$ of positive numbers tending to $0$ such that, as $k\rightarrow \I$,
$$
\delta_k u_{\delta_k}(x_0)\rightarrow \lambda^*
$$
and $u_{\delta_k}-u_{\delta_k}(x_0)$ converges to a solution $u^*$ of \eqref{lamPDE} in $C^1_{\text{loc}}(\R^n)$ that satisfies \eqref{ugrowth}.
\end{thm}

\par Along the way to proving Theorem \ref{mainthm}, we will make a few interesting observations.  The first, which was already known, is that when $f$ is rotational $u_\delta$ (and thus $u^*)$ are also rotational
and in particular of class $C^2(\R^n)$.  This result was first proved in \cite{Lu}.  However, our argument is elegant and does not require an explicit construction of the solution. 

\begin{thm}\label{RotReg}
Suppose that $f$ is rotationally symmetric i.e. $f(x)=f_0(|x|)$ for some non-decreasing, convex, superlinear function $f_0: [0,\I)\rightarrow \R.$  Then $u_\delta$ is rotationally symmetric
and belongs to the space $C^2(\R^n)$.
\end{thm}
Our second observation, which is original, is that the eigenvalue $\lambda^*$ can be characterized via a min-max formula.

\begin{thm}\label{MinMaxThm}Define 
$$
\lambda_-=\sup\left\{\inf_{x\in\R^n}\left\{\Delta \phi(x) + f(x)\right\} :  \phi\in C^2(\R^n), |D\phi|\le 1 \right\}
$$
and 
$$
\lambda_+=\inf\left\{\sup_{|D\psi(x)|<1}\left\{\Delta \psi(x) + f(x)\right\} :  \psi\in C^2(\R^n), \liminf_{|x|\rightarrow \infty}\frac{\psi(x)}{|x|}\ge 1 \right\}.
$$
(i) Then
$$
\lambda_-=\lambda^*\le \lambda_+.
$$
(ii) Moreover, if there is a $C^2(\R^n)$ supersolution $\psi^*$ of \eqref{lamPDE} with eigenvalue $\lambda^*$, such that $$\liminf_{|x|\rightarrow \infty}\frac{\psi^*(x)}{|x|}\ge 1,$$ then $\lambda^*=\lambda_+.$
\end{thm}

\par As far as we know, this work is the first to consider the eigenvalue problem as posed above.  However, a big part of our motivation was the work of Menaldi et. al. \cite{M} who studied a closely related problem arising in stochastic control theory.  
With regards to the framework we present,  they used probabilistic arguments to build an eigenvalue. In this paper, we establish the existence of a {\it unique} eigenvalue and obtain a better regularity result than what was obtained in \cite{M} 
as it does not require any special assumptions on $f$; we only require convexity and superlinear growth.  Moreover, we have employed methods that are entirely analytic and use nothing from probability theory.   Before, carrying out our analysis let us see how this eigenvalue problem arises in the theory of singular stochastic control.
\newline 
\newline
\noindent
{\bf Probabilistic interpretation of the eigenvalue}.  Let $(\Omega, {\mathcal F}, \mathbb{P})$ be 
a probability space with $n$-dimensional Brownian motion $(W(t),t\ge0)$. Also set
$$
X^\nu(t):=\sqrt{2} W(t) + \nu(t), \; t\ge 0
$$
where  $\nu$ is an $\R^n$ valued control process (adapted to 
the filtration generated by $W$) that satisfies

$$
\begin{cases}
\nu(0)=0\;\text{a.s.}\\
t\mapsto \nu(t) \; \text{is left continuous a.s.}\\
|\nu|(t):=TV_{\nu}[0,t)<\infty, \;\text{for all}\; t> 0\;\text{ a.s.} 
\end{cases}.\footnote{$TV_g[a,b)$ denotes the total variation of $g$ on the interval $[a,b)$.}
$$
We say $\nu$ is a {\it singular control} as it may have sample paths that may not be absolutely continuous with respect to Lebesgue measure on $[0,\infty)$. 

\par  An optimization problem of interest is to find a singular control $\nu$ that minimizes the quantity
\begin{equation}\label{ergQ}
\limsup_{t\rightarrow \infty}\frac{1}{t}\left\{\E\int^{t}_{0}f(X^\nu(s))ds + |\nu|(t)\right\}.
\end{equation}
As \eqref{ergQ} is a ``long-time" average, we interpret this problem as one of {\it singular ergodic control}.

\par To see how \eqref{lamPDE} is related to the control problem described above, we suppose that there is $\lambda\in\R$ such that
\eqref{lamPDE} has a convex solution $u\in C^2(\R^n)$.  Let $\nu$ be a singular control process. According to Ito's rule for semi-martingales (pp. 278-301 \cite{M}),
\begin{eqnarray}
\E u(X^\nu(t)) & = & u(x) + \E\int^{t}_{0}\Delta u(X^\nu(s))ds + \E\int^{t}_{0}Du(X^\nu(s))\cdot d\nu(s) \nonumber \\
                         &   & + \sum_{0\le s<t}\E\int^{t}_{0}\left[u(X^\nu(s+))-u(X^\nu(s)) - Du(X^\nu(s))\cdot (X^\nu(s+) - X^\nu(s)) \right] \nonumber\\
                         &\ge & u(x) +  t\lambda - \E\int^{t}_{0}f(X^\nu(s))ds - \E\int^{t}_{0}|Du(X^\nu(s))| d|\nu|(s) \nonumber \\
                          &\ge & u(x) + t\lambda - \E\int^{t}_{0}f(X^\nu(s))ds - |\nu|(t) \nonumber.
\end{eqnarray}
Thus
\begin{equation}\label{passlimit?}
\lambda\le \frac{1}{t}\left\{\E\int^{t}_{0}f(X^\nu(s))ds + |\nu|(t)\right\} + \frac{\E u(X^\nu(t)) -u(x) }{t}, \quad t>0.
\end{equation}
\par We would like to conclude that
\begin{equation}\label{lamlessthan}
\lambda\le \limsup_{t\rightarrow \infty}\frac{1}{t}\left\{\E\int^{t}_{0}f(X^\nu(s))ds + |\nu|(t)\right\}.
\end{equation}
Suppose that the right hand side of the inequality \eqref{lamlessthan} is finite or else \eqref{lamlessthan} clearly holds. In this case,
$$
\limsup_{t\rightarrow \infty}\frac{1}{t}\int^{t}_{0}\E f(X^\nu(s))ds<\infty
$$
which implies that there is a sequence of positive numbers $t_k\rightarrow \infty$ as $k\rightarrow \infty$ such that 
$$
\limsup_{k\rightarrow \infty}\E f(X^\nu(t_k))<\infty.
$$
As $f$ grows superlinearly and as $u$ grows at most as fast as $|x|$, as $|x|\rightarrow \infty$,
$$
\limsup_{k\rightarrow \infty}\E u(X^\nu(t_k))<\infty.
$$
Choosing $t=t_k$ in \eqref{passlimit?} and sending $k\rightarrow \infty$ establishes \eqref{lamlessthan}.  In particular,
\begin{equation}\label{lamlessthanlamstar}
\lambda\le \lambda^*:=\inf_{\nu}\limsup_{t\rightarrow \infty}\frac{1}{t}\left\{\E\int^{t}_{0}f(X^\nu(s))ds + |\nu|(t)\right\}.
\end{equation}

\par If there is a control $\nu^*$ such that 
\begin{equation}\label{nu1}
\lambda - \Delta u(X^{\nu^*}(t)) - f(X^{\nu^*}(t))=0, \; \text{for $t\in(0,\infty)$ \text{a.s.}},
\end{equation}

\begin{equation}\label{nu2}
\int^{t}_{0}Du(X^{\nu^*}(s))\cdot d\nu^*(s)=-\nu^*(t),\; \text{for $t\in(0,\infty)$ \text{a.s.}},
\end{equation}
and

\begin{equation}\label{nu3}
u(X^{\nu^*}(t+))-u(X^{\nu^*}(t)) - Du(X^{\nu^*}(t))\cdot (X^{\nu^*}(t+) - X^{\nu^*}(t))=0,\; \text{for $t\in(0,\infty)$ \text{a.s.}},
\end{equation}
then equality holds in \eqref{lamlessthanlamstar}.  In this case, we have $\lambda^*$ as a probabilistic formula for the eigenvalue appearing 
in \eqref{lamPDE}.   
\begin{rem}
$\nu^*$ satisfying \eqref{nu1}, \eqref{nu2}, and \eqref{nu3} is a good candidate for an optimal control. Designing such an optimal control can be done via reflected diffusions if enough regularity 
is assumed on $u$ and on the boundary of the set of points $x$ such that $|Du(x)|<1$. This procedure is discussed in detail in section 5 of \cite{M} and in section 12 of \cite{SS}.
\end{rem}

\section{Comparison of eigenvalues}\label{CompEig}

The purpose of this section is to prove that there can be at most one eigenvalue for which the PDE \eqref{lamPDE} has a viscosity solution satisfying 
\eqref{ugrowth}.  We will not motivate viscosity solutions except to say that it is a weak notion of 
solutions of PDE and that it provides a very flexible approach to establishing comparison
principles for sub- and supersolutions of scalar, non-linear, elliptic and parabolic PDE. We refer the interested reader to a few of the well known references for
the theory of viscosity solutions such as \cite{CIL,FS}.

\begin{defn}\label{ViscEigDef}
$u\in USC(\R^n)$ is a {\it viscosity subsolution} of $\eqref{lamPDE}$ with {\it eigenvalue $\lambda\in\R$}
if for each $x_0\in\R^n$, 
$$
\max\left\{ \lambda -\Delta \varphi(x_0) - f(x_0), |D\varphi(x_0)|-1\right\}\le 0
$$
whenever $u-\varphi$ has a local maximum at $x_0$ and $\varphi \in C^2(\R^n)$. $v\in LSC(\R^n)$ is a {\it viscosity supersolution}
of $\eqref{lamPDE}$ with {\it eigenvalue $\mu\in\R$} if for each $y_0\in\R^n$, 
$$
\max\left\{ \mu -\Delta \psi(y_0) - f(y_0), |D\psi(y_0)|-1\right\}\ge 0
$$
whenever $v-\psi$ has a local minimum at $y_0$ and $\psi \in C^2(\R^n)$.  $u\in C(\R^n)$ is a {\it viscosity 
solution} of \eqref{lamPDE} with {\it eigenvalue} $\lambda\in\R$ if it is both a viscosity sub- and supersolution of \eqref{lamPDE} with
eigenvalue $\lambda.$
\end{defn}
An equivalent definition can be given via second order sub- and super-jets. 

\begin{defn}
$(i)$ Let $x_0\in \R^n$. $(p,X)\in \R^n\times {\mathcal S(n)}$ belongs to the {\it second order superjet of $u$ at $x_0$} if 
\begin{equation}\label{superjet}
u(x)\le u(x_0) + p\cdot(x-x_0) + \frac{1}{2}X(x-x_0)\cdot (x-x_0)+ o(|x-x_0|^2)
\end{equation}
as $|x-x_0|\rightarrow 0.$\footnote{${\mathcal S}(n)$ is the set of $n\times n$ symmetric matrices. } The collection of all such pairs $(p,X)$ is denoted $J^{2,+}u(x_0)$. \\
$(ii)$ Let $x_0\in \R^n$. $(p,X)\in \R^n\times {\mathcal S(n)}$ belongs to the {\it second order subjet of $u$ at $x_0$} if 
\begin{equation}\label{subjet}
u(x)\ge u(x_0) + p\cdot(x-x_0) + \frac{1}{2}X(x-x_0)\cdot (x-x_0)+ o(|x-x_0|^2)
\end{equation}
as $|x-x_0|\rightarrow 0. $ The collection of all such pairs $(p,X)$ is denoted $J^{2,-}u(x_0)$. 
\end{defn}
\begin{rem}
First order sub- and super-jets  are defined similarly. 
\end{rem}
Notice that if $u-\varphi$ has a local maximum [minimum] at $x_0$ and $\varphi$ is smooth, then \eqref{superjet} [\eqref{subjet}] holds with 
\begin{equation}\label{jet2varphi}
p=D\varphi(x_0)\quad \text{and}\quad X=D^2\varphi(x_0).
\end{equation}
A converse of this fact is also true and we refer the reader to \cite{CIL} for a proof. 

\begin{lem}
Suppose that $(p,X)\in J^{2,+}u(x_0)$. Then there is an open set $U\ni x_0$ and $\varphi\in C^2(U)$ such that \eqref{jet2varphi} holds. 
\end{lem}
Now its clear that $u\in USC(\R^n)$ is a viscosity subsolution of $\eqref{lamPDE}$ with eigenvalue $\lambda\in\R$
if for each $(p,X)\in J^{2,+}u(x_0)$, 
\begin{equation}\label{jetSub}
\max\left\{ \lambda -\tr X - h(x_0), |p|-1\right\}\le 0;
\end{equation}
and $v\in LSC(\R^n)$ is a viscosity super-solution of $\eqref{lamPDE}$ with eigenvalue $\mu\in\R$
if for each $(q,Y)\in J^{2,-}v(y_0)$, 
\begin{equation}\label{jetSuper}
\max\left\{ \mu -\tr Y - h(y_0), |q|-1\right\}\ge 0.
\end{equation}

\begin{rem}
We can actually get away with less when determining whether or not a function is a sub- or supersolution.  Define
\begin{eqnarray}
\overline{J}^{2,+}u(x_0)&=&\{(p,X)\in \R^n \times {\mathcal S}(n) : \; \text{there exists} \; (x_n, p_n,X_n) \in \R^n \times J^{2,+}u(x_n),\text{ for } n\in\N, \nonumber \\
&& \text{such that} \; (x_n,u(x_n), p_n, X_n)\rightarrow (x_0,u(x_0),p,X), \; \text{as}\; n\rightarrow \infty\} \nonumber
\end{eqnarray}
and
\begin{eqnarray}
\overline{J}^{2,-}v(y_0)&=&\{(q,Y)\in \R^n \times {\mathcal S}(n) : \; \text{there exists} \; (y_n, q_n,Y_n) \in \R^n \times J^{2,-}v(y_n),\text{ for } n\in\N, \nonumber \\
&& \text{such that} \; (y_n,v(y_n), q_n, Y_n)\rightarrow (y_0,v(y_0),q,Y), \; \text{as}\; n\rightarrow \infty\}. \nonumber
\end{eqnarray}
It is straightforward to show that if $u$ is a viscosity subsolution of \eqref{lamPDE} with eigenvalue $\lambda$ and $(p,X)\in \overline{J}^{2,+}u(x_0)$, then \eqref{jetSub} still holds; and if $v$ is a viscosity supersolution of \eqref{lamPDE} with eigenvalue $\mu$ and $(q,Y)\in \overline{J}^{2,-}v(y_0)$, then \eqref{jetSuper} still holds.
\end{rem}
Towards establishing a uniqueness result, we first establish a comparison principle for eigenvalues with sub- and supersolutions.  As we shall do several times 
in this paper, we will first present a formal (heuristic) argument followed by a rigorous proof.  The purpose of doing this is to convey the motivating ideas.

\begin{prop}\label{comparison}
Suppose $u$ is a subsolution of \eqref{lamPDE} with eigenvalue $\lambda$ and that 
 $v$ is a supersolution of \eqref{lamPDE} with eigenvalue $\mu$. If in addition
\begin{equation}\label{BCcomp1}
\limsup_{|x|\rightarrow \infty}\frac{u(x)}{|x|}\le 1\le \liminf_{|x|\rightarrow \infty}\frac{v(x)}{|x|},
\end{equation}
then $\lambda\le \mu.$
\end{prop}

\noindent 
{\it Formal proof}. Here we assume that $u,v\in C^2(\R^n)$. Fix $0<\epsilon<1$ and set 
$$
w^\epsilon(x)=\epsilon u(x) - v(x), \quad x\in\R^n.
$$
By \eqref{BCcomp1}, we have $\lim_{|x|\rightarrow \infty}w^\epsilon(x)=-\infty$, so there is $x_\epsilon\in\R^n$ such that
$$
w^\epsilon(x_\epsilon)=\sup_{x\in\R^n}w^\epsilon(x).
$$
Basic calculus gives
$$
\begin{cases}
0=Dw^\epsilon(x_\epsilon)=\epsilon Du(x_\epsilon)-Dv(x_\epsilon) \\
0\ge \Delta w^\epsilon(x_\epsilon)=\epsilon \Delta u(x_\epsilon) - \Delta v(x_\epsilon) \\
\end{cases}.
$$
Note in particular that
$$
|Dv(x_\epsilon)|=\epsilon |Du(x_\epsilon)|\le \epsilon <1,
$$
and since $v$ is a supersolution of \eqref{lamPDE} with eigenvalue $\mu$,
$$
0\le \mu - \Delta v(x_\epsilon)-f(x_\epsilon).
$$
As $u$ is a subsolution of  \eqref{lamPDE} with eigenvalue $\lambda$
\begin{eqnarray}
\epsilon \lambda - \mu&\le& \epsilon \Delta u(x_\epsilon) - \Delta v(x_\epsilon)-(1-\epsilon)f(x_\epsilon) \nonumber \\
 &\le & -(1-\epsilon)f(x_\epsilon) \nonumber \\
 &\le & 0 \nonumber.
\end{eqnarray}
Here we have used that $f$ is non-negative. Letting $\epsilon\rightarrow 1^-$, gives $\lambda\le \mu$.  
\begin{flushright}
$\Box$
\end{flushright}
We now make this rigorous by using a ``doubling the variables" type of argument.  This is a fairly standard approach in the 
theory of viscosity solutions. What makes this problem different is the fact that the domain is the whole space $\R^n$. However, we 
have the growth condition \eqref{ugrowth} which plays the role of a boundary condition.

\begin{proof} (of the proposition) 
1. Fix $0<\epsilon <1$ and set 
$$
w^\epsilon(x,y)=\epsilon u(x) - v(y), \quad x,y\in\R^n.
$$
For $\delta >0$, we also set 
$$
\varphi_\delta(x,y)=\frac{1}{2\delta}|x-y|^2, \quad x,y\in\R^n.
$$
The inequality 
\begin{eqnarray}\label{wepsIneq}
w^\epsilon(x,y)-\varphi_\delta(x,y) & = &\epsilon (u(x)-u(y)) - \frac{1}{2\delta}|x-y|^2 + \epsilon u(y)-v(y) \nonumber \\
                                                              &\le &\left(|x-y| - \frac{1}{2\delta}|x-y|^2\right) +  \epsilon u(y)-v(y) \nonumber
\end{eqnarray}
implies 
$$
\lim_{|(x,y)|\rightarrow \infty}\left\{w^\epsilon(x,y)-\varphi_\delta(x,y) \right\} = -\infty.
$$
Therefore, $w^\epsilon-\varphi_\delta$ achieves a global maximum at a point $(x_\delta, y_\delta)\in \R^n\times\R^n$.  

\par 2. According to the Theorem of Sums (Theorem 3.2 in \cite{CIL}), for each $\rho>0$, there are $X,Y\in {\mathcal S}(n)$ such that 

$$
\left(\frac{x_\delta - y_\delta}{\delta}, X\right)=\left(D_x\varphi_\delta(x_\delta,y_\delta), X\right)\in \overline{J}^{2,+}(\epsilon u)(x_\delta),
$$ 
\begin{equation}\label{vCalcCond}
\left(\frac{x_\delta - y_\delta}{\delta}, Y\right)=\left(-D_y\varphi_\delta(x_\delta,y_\delta), Y\right)\in \overline{J}^{2,-}v(y_\delta),
\end{equation}
and 
\begin{equation}\label{ImportantMatIneq}
\left(\begin{array}{cc}
X & 0 \\
0 & - Y
\end{array}\right)\le A +\rho A^2.
\end{equation}
Here 
$$
A=D^2\varphi_\delta(x_\delta,y_\delta)=\frac{1}{\delta}
\left(\begin{array}{cc}
I_n & -I_n \\
-I_n & I_n
\end{array}\right).
$$
Applying both sides of the matrix inequality \eqref{ImportantMatIneq} to the vector $(\xi, \xi)^t\in \R^{2n}$ and then taking the dot product with $(\xi, \xi)^t$ yields
$$
X\xi\cdot \xi - Y\xi\cdot \xi\le 0.
$$ 
As $\xi\in \R^n$ is arbitrary, $X\le Y$.

\par 3. We also have 
$$
\frac{1}{\epsilon}\frac{x_\delta - y_\delta}{\delta}\in \overline{J}^{1,+}u(x_\delta),
$$
and since $|Du|\le 1$ (in the sense of viscosity solutions),
$$
\left|\frac{x_\delta - y_\delta}{\delta}\right|\le \epsilon <1. 
$$
Since $v$ is a viscosity supersolution of \eqref{lamPDE} with eigenvalue $\mu$, we have
$$
0\le \mu -\tr Y - f(y_\delta)
$$
by  \eqref{vCalcCond}. As $u$ is a viscosity subsolution of \eqref{lamPDE} with eigenvalue $\lambda$,
$$
\lambda -\frac{\tr{X}}{\epsilon} - f(x_\delta)\le 0.
$$
Therefore,
\begin{equation}\label{lammuest}
\epsilon \lambda -\mu \le \tr[X-Y]+\epsilon f(x_\delta)-f(y_\delta)\le f(x_\delta)-f(y_\delta).
\end{equation}

\par 4. We now claim that $x_\delta\in\R^n$ is bounded for all small enough $\delta>0.$ If not, then there is a sequence of $\delta\rightarrow 0$, for which
$(w^{\epsilon}-\varphi_\delta)(x_\delta,y_\delta)$ tends to $-\infty$ as this sequence of $\delta$ tends to 0.  Indeed 
\begin{eqnarray}
(w^{\epsilon}-\varphi_\delta)(x_\delta,y_\delta)&=& (\epsilon u(x_\delta)-v(x_\delta) ) + v(x_\delta)-v(y_\delta) -\frac{|x_\delta - y_\delta|^2}{2\delta}\nonumber \\
&\le & (\epsilon u(x_\delta)-v(x_\delta) ) + |x_\delta-y_\delta|-\frac{|x_\delta - y_\delta|^2}{2\delta}\nonumber \\ 
&\le & \epsilon u(x_\delta)-v(x_\delta)  +\frac{\delta}{2} \nonumber
\end{eqnarray}
which tends to $-\infty$ as $\delta \rightarrow 0$ provided $\lim_{\delta\rightarrow 0^+}|x_\delta|=+\infty.$ This would be the case for some sequence of  $\delta\rightarrow 0$,
provided $x_\delta $ is unbounded. 

\par However,
\begin{eqnarray}
(w^{\epsilon}-\varphi_\delta)(x_\delta,y_\delta)&=&\max_{x,y\in\R^n}\left\{\epsilon u(x) - v(y) -\frac{|x-y|^2}{2\delta}\right\}\nonumber \\
&\ge & \epsilon u(0)-v(0) \nonumber \\
&>&-\infty, \nonumber
\end{eqnarray}
and thus $x_\delta$ lies in a bounded subset of $\R^n$.  Likewise, we conclude that $y_\delta$ is also bounded for all $\delta>0$ and small.  It then follows from Lemma 3.1 in \cite{CIL}, 
$$
\lim_{\delta\rightarrow 0^+}\frac{|x_\delta-y_\delta|^2}{2\delta}=0,
$$
and therefore the sequence $( (x_\delta,y_\delta) )_{\delta>0}$ has a cluster point $(x_\epsilon,x_\epsilon)$ for some sequence of $\delta\rightarrow 0^+.$ Passing to this limit in \eqref{lammuest} along this such a sequence gives
$$
\epsilon \lambda -\mu\le 0.
$$
We conclude by letting $\epsilon\rightarrow 1^-$.
\end{proof}
Uniqueness of eigenvalues with solutions having the appropriate growth for large values of $|x|$ is now immediate.
 
\begin{cor}\label{uniquenessoflam}
There can be at most one $\lambda \in \R$ such that \eqref{lamPDE} has a viscosity solution $u$ satisfying the growth condition
\eqref{ugrowth}.
\end{cor}

\section{Approximation scheme}\label{EllipEst}
Another interesting corollary of Proposition \ref{comparison} is

\begin{cor}\label{minmaxformula}
Suppose there exists an eigenvalue $\lambda^*$ as described in Theorem \ref{mainthm}. Then 
\begin{eqnarray}\label{maxform}
\lambda^*&=&\sup\text{{\huge\{} } \lambda \in\R: \text{there exists a subsolution $u$ of \eqref{lamPDE} with eigenvalue $\lambda$}, \nonumber \\
&& \hspace{2in}\left. \text{satisfying $\limsup_{|x|\rightarrow \infty}\frac{u(x)}{|x|}\le 1$.}  \right\}
\end{eqnarray}
and
\begin{eqnarray}\label{minform}
\lambda^*&=&\inf\text{{\huge\{} } \mu \in\R: \text{there exists a supersolution $v$ of \eqref{lamPDE} with eigenvalue $\mu$}, \nonumber \\
&& \hspace{2in}\left. \text{satisfying $\liminf_{|x|\rightarrow \infty}\frac{v(x)}{|x|}\ge 1$.}  \right\}
\end{eqnarray}
\end{cor}
It would be of great interest to show both expressions on the right hand sides of \eqref{maxform} and \eqref{minform} are equal and that this number is the unique eigenvalue of
equation \eqref{lamPDE}. Such a procedure for producing eigenvalues would be reminiscent of Perron's method for exhibiting viscosity solutions of PDE enjoying a comparison principle. 
Unfortunately, this method does not work so directly in our context as it is not clear that if, say, the right hand side of \eqref{maxform} is not an eigenvalue we can find a strictly bigger number with 
a corresponding $u$ that is a subsolution of \eqref{lamPDE}. Therefore, we are led to an alternative procedure of approximating an eigenvalue.

\par The method we propose is a PDE version of the probabilistic approach used by Menaldi et.al \cite{M}. However, we believe the earliest application of this method appears in periodic homogenization of 
viscosity solutions of PDE in \cite{LPV,E}.  This approach essentially amounts to replacing $\lambda$ in \eqref{lamPDE} with ``$\delta u$'" and studying the resulting PDE for $\delta>0$ and small. 
If this resulting PDE has a unique solution $u_\delta$, the hope is that there is a sequence of $\delta$ tending to 0 such that  $\delta u_\delta$ tends to $\lambda.$

\par To this end, we will study solutions of the PDE \eqref{deltaPDE}

$$
\max\left\{\delta u  - \Delta u  - f, |Du| -1 \right\}=0, \quad x\in\R^n.
$$
In particular, we seek a viscosity solution $u$ satisfying \eqref{ugrowth}

$$
\lim_{|x|\rightarrow \infty}\frac{u(x)}{|x|}=1. \nonumber
$$

\begin{prop}\label{comparisonU}
Suppose $u$ is a subsolution of \eqref{deltaPDE} and that 
 $v$ is a supersolution of \eqref{deltaPDE}.  If in addition
$$
\limsup_{|x|\rightarrow \infty}\frac{u(x)}{|x|}\le 1\le \liminf_{|x|\rightarrow \infty}\frac{v(x)}{|x|},
$$
then $u\le v.$
\end{prop}

\begin{proof}
We omit the proof as it is almost identical to the proof of Proposition \ref{comparison}. 
\end{proof}

\begin{cor}\label{uniquenessofU}
There can be at most one viscosity solution $u$ \eqref{deltaPDE} satisfying the growth condition \eqref{ugrowth}.
\end{cor}

With a comparison principle in hand, we can now employ a routine application of Perron's method to obtain existence of solutions once we have appropriate sub and supersolutions. 

\begin{lem}\label{sarm} Fix $0<\delta\le 1$.\\
$(i)$ There is a universal constant $K>0$ such that  

\begin{equation}\label{viscsub}
\underline{u}(x)=\left(|x|- K\right)^+,\; x\in\R^n
\end{equation}
is a viscosity subsolution of \eqref{deltaPDE} satisfying the growth condition \eqref{ugrowth}. \\
$(ii)$ There is a universal constant $K>0$ such that 

\begin{equation}\label{viscsuper}
\overline{u}(x)=\frac{K}{\delta} + \begin{cases}\frac{1}{2}|x|^2,\;\;\; |x|\le 1\\  |x|-\frac{1}{2},|x|\ge 1\end{cases} ,\; x\in\R^n
\end{equation}
is a viscosity supersolution of  \eqref{deltaPDE} satisfying the growth condition \eqref{ugrowth}.
\end{lem}

\begin{proof} $(i)$ Choose $K>0$ such that 
$$
\underline{u}(x)\le f(x), \quad x\in \R^n.
$$
Such a $K$ can be chosen due to the superlinear growth of $f$. 

\par As $\overline{u}$ is convex and as Lip($\underline{u}$)=1, if $(p,X)\in J^{2,+}\underline{u}(x_0)$
$$
|p|\le 1\quad \text{and}\quad X\ge 0.
$$
Hence, 
$$
\max\left\{\delta \underline{u}(x_0)-\tr X -f(x_0), |p|-1\right\}\le \max\left\{\underline{u}(x_0) - f(x_0), |p|-1\right\}\le 0.
$$
Thus $\underline{u}$ is a viscosity subsolution. 

\par $(ii)$ Choose 
$$
K:=n+\max_{|x|\le 1}f(x)
$$
and assume that $(p,X)\in J^{2,-}\underline{u}(x_0)$. If $|x_0|<1$, $\bar{u}$ is smooth in a neighborhood of $x_0$ and

$$
\begin{cases}
\bar{u}(x_0)=\frac{K}{\delta} +\frac{|x_0|^2}{2}\\
D\bar{u}(x_0)=x_0=p\\
\Delta\bar{u}(x_0)=n=\tr X
\end{cases}.
$$
Therefore,
$$
\delta \bar{u}(x_0) - \Delta \bar{u}(x_0) - h(x_0)\ge K - n - f(x_0)\ge 0,
$$
which implies
\begin{equation}\label{baruSupersoln}
\max\left\{\delta \bar{u}(x_0) - \Delta \bar{u}(x_0) - f(x_0), |D\bar{u}(x_0)|- 1\right\}\ge 0.
\end{equation}
Now suppose $|x_0|\ge 1$. $\bar{u}\in C^{1}(\R^n)$, so $p=D\bar{u}(x_0)=x_0/|x_0|$ and in particular $|D\bar{u}(x_0)|=1$. Thus
\eqref{baruSupersoln} still holds, and consequently,  $\bar{u}$ is a viscosity supersolution.
\end{proof}
The following proposition, which implies part $(ii)$ of Theorem \ref{deltathm}, follows directly from Theorem 4.1 in \cite{CIL} using $\underline{u}$ and $\bar{u}$ above. 
As the proof is a routine application of Perron's method, we omit it.

\begin{prop}
Fix $0<\delta \le 1$. There is a unique viscosity solution $u=u_\delta$ of the \eqref{deltaPDE} satisfying \eqref{ugrowth}.
\end{prop}

\subsection{Basic estimates}\label{RegQuote}
\par Before we attempt to pass to the limit as $\delta\rightarrow 0$, we will obtain some better estimates on $u_\delta$ that will help us build an eigenvalue $\lambda^*$ and establish
estimates on a solution $u^*$ of \eqref{lamPDE} corresponding to this eigenvalue.   So far we have shown that \eqref{deltaPDE} has a unique solution $u_\delta$ that satisfies the growth condition \eqref{ugrowth}. Moreover, from the sub- and supersolutions \eqref{viscsub} and \eqref{viscsuper} above, we have for each $0<\delta\le 1$

$$
(|x|-K)^+\le u_\delta(x)\le \frac{K}{\delta}+|x|, \; x\in\R^n
$$
and 
$$
|u_\delta(x)-u_\delta(y)|\le |x-y|,\; x,y\in\R^n.
$$
Our goal now is to obtain {\it second derivative estimates} on $u_\delta$.  We first prove

\begin{prop}\label{Basic2ndDer}
There is a constant $C>0$ such that for all $0<\delta\le 1$ and (Lebesgue) almost every $x\in\R^n$, the following estimate holds
\begin{equation}\label{Basic2ndDerEst}
0\le D^2u_\delta(x)\le \frac{1}{\delta}\max_{|y|\le C}|D^2f(y)|.
\end{equation}
\end{prop}
The above proposition follows from the following two lemmas. In the first lemma, we show that $u_\delta$ is convex by adapting the classical ``convexity maximum principle" 
argument of Korevaar \cite{KO}; in the second lemma, we estimate the second-order difference quotient of $u_\delta$ from above to obtain the upper bound in \eqref{Basic2ndDerEst}.  

\begin{lem}\label{UDeltaConvex}
$u_\delta$ is convex. 
\end{lem}
\begin{proof}
1.  We first assume $u\in C^2(\R^n)$ and for ease of notation, we write $u$ for $u_\delta $. Fix  $0< \epsilon <1$ and set 
$$
{\mathcal C}^\epsilon(x,y)=\epsilon u\left(\frac{x+y}{2}\right)-\frac{u(x) + u(y)}{2}, \quad x,y\in\R^n.
$$
We aim to bound ${\mathcal C}^\epsilon$ from above and later send $\epsilon\rightarrow 1^-.$ 

\par We first claim that ${\mathcal C}^\epsilon$ achieves its maximum value at some point $(x_\epsilon, y_\epsilon)\in \R^n\times\R^n$; it suffices
to show 
\begin{equation}\label{ConVexUClaim}
\lim_{|(x,y)|\rightarrow \infty}{\mathcal C}^\epsilon(x,y)=-\infty.
\end{equation}
Let $(x_k,y_k)\in \R^n\times \R^n$ be such that 

$$
|x_k|+|y_k|\rightarrow \infty
$$
as $k\rightarrow \infty$. 
Let $N$ be large enough so that $|x_k|+|y_k|>0$ for $k\ge N$. Note that for $k\ge N$

\begin{eqnarray}
\frac{{\mathcal C}^\epsilon(x_k,y_k)}{|x_k|+|y_k|} &=& \epsilon \frac{u\left(\frac{x_k+y_k}{2}\right)}{|x_k|+|y_k|} -\frac{1}{2}\left\{\left(\frac{|x_k|}{|x_k|+|y_k|}\right)\frac{u(x_k)}{|x_k|}+ 
\left(\frac{|y_k|}{|x_k|+|y_k|}\right)\frac{u(y_k)}{|y_k|}\right\} \nonumber \\
&\le & \frac{\epsilon}{2} \frac{u\left(\frac{x_k+y_k}{2}\right)}{\left|\frac{x_k+y_k}{2}\right|} -\frac{1}{2}\left\{\left(\frac{|x_k|}{|x_k|+|y_k|}\right)\frac{u(x_k)}{|x_k|}+ 
\left(\frac{|y_k|}{|x_k|+|y_k|}\right)\frac{u(y_k)}{|y_k|}\right\} \nonumber
\end{eqnarray}
when of course  $|x_k+y_k|>0.$

\par If $|x_k+y_k|$ happens to be bounded, then

$$
\limsup_{k\rightarrow \infty }\frac{{\mathcal C}^\epsilon(x_k,y_k)}{|x_k|+|y_k|} \le -\frac{1}{2}<0.
$$
While if  $|x_k+y_k|\rightarrow \infty$, as $k\rightarrow \infty$, we still have
$$
\limsup_{k\rightarrow \I}\frac{{\mathcal C}^\epsilon(x_k,y_k)}{|x_k|+|y_k|} \le \frac{-(\epsilon -1)}{2}<0.
$$
Consequently, $\limsup_{k\rightarrow\infty }{\mathcal C}^\epsilon(x_k,y_k)=-\infty$. The claim \eqref{ConVexUClaim} follows since $(x_k,y_k)$ was an arbitrary unbounded sequence.

\par 2. As $(x_\epsilon,y_\epsilon)$ is a maximizing point for ${\mathcal C}^\epsilon$,
$$
0=D_x{\mathcal C}^\epsilon(x_\epsilon, y_\epsilon)=\frac{\epsilon}{2}Du\left(\frac{x_\epsilon+y_\epsilon}{2}\right)-\frac{1}{2}Du(x_\epsilon)
$$
and
$$
0=D_y{\mathcal C}^\epsilon(x_\epsilon, y_\epsilon)=\frac{\epsilon}{2}Du\left(\frac{x_\epsilon+y_\epsilon}{2}\right)-\frac{1}{2}Du(y_\epsilon).
$$
Thus,
\begin{equation}\label{firstorderC}
\epsilon Du\left(\frac{x_\epsilon+y_\epsilon}{2}\right)=Du(x_\epsilon)=Du(y_\epsilon). \nonumber
\end{equation}
The function $v\mapsto {\mathcal C}^\epsilon(x_\epsilon + v, y_\epsilon + v)$ has a maximum at $v=0$ which implies

\begin{equation}\label{secondorderC}
0\ge \epsilon \Delta u\left(\frac{x_\epsilon+y_\epsilon}{2}\right)- \frac{\Delta u(x_\epsilon)+\Delta u(y_\epsilon)}{2}. \nonumber
\end{equation}
Since, 
$$
|Du(x_\epsilon)|=|Du(y_\epsilon)|= \epsilon\left|Du\left(\frac{x_\epsilon+y_\epsilon}{2}\right)\right|\le \epsilon <1,
$$
we have
$$
\delta u(z) - \Delta u(z) - f(z)=0,\quad z=x_\epsilon, y_\epsilon.
$$
Combining the above inequalities gives
\begin{eqnarray}
\delta {\mathcal C}^\epsilon(x,y)&\le &\delta {\mathcal C}^\epsilon(x_\epsilon, y_\epsilon)  \nonumber\\
&= &\epsilon \delta u\left(\frac{x_\epsilon+y_\epsilon}{2}\right)-\frac{\delta u(x_\epsilon) + \delta u(y_\epsilon)}{2} \nonumber \\
&\le &\epsilon \Delta u\left(\frac{x_\epsilon+y_\epsilon}{2}\right)-\frac{\Delta u(x_\epsilon) + \Delta u(y_\epsilon)}{2} \nonumber \\
&& \epsilon f\left(\frac{x_\epsilon + y_\epsilon}{2}\right)-\frac{f(x_\epsilon) + f(y_\epsilon)}{2} \nonumber \\
&\le &f\left(\frac{x_\epsilon + y_\epsilon}{2}\right)-\frac{f(x_\epsilon) + f(y_\epsilon)}{2} \nonumber \\
&\le& 0 \nonumber
\end{eqnarray}
by the convexity of $f$, for each $x,y\in\R^n.$ Sending $\epsilon\rightarrow 1^-$, we conclude that $u$ is convex.

\par 3. To make this formal argument rigorous, we employ a doubling the variables type of argument. Moreover, since ${\mathcal C}^\epsilon$ 
above is a type of doubling the variables function, it is appropriate to ``quadruple the variables." This can be done by fixing $0< \epsilon  <1$ and setting
$$
w^\epsilon(x,y,x',y')=\epsilon u\left(\frac{x+y}{2}\right)-\frac{u(x') + u(y')}{2}, \quad x,y,x', y'\in\R^n,
$$
and for $\eta>0$, setting

$$
\varphi_\eta(x,y,x',y')=\frac{1}{2\eta}\left\{ |x-x'|^2 + |y-y'|^2\right\}, \quad x,y,x', y'\in\R^n.
$$
Notice that
\begin{eqnarray}
(w^\epsilon -\varphi_\eta)(x,y,x',y')&=& \epsilon \left\{u\left(\frac{x+y}{2}\right) - u\left(\frac{x'+y'}{2}\right)\right\} - \frac{1}{2\eta}\left\{ |x-x'|^2 + |y-y'|^2\right\}\nonumber \\
                                                             &  & + \epsilon u\left(\frac{x'+y'}{2}\right) - \frac{u(x') + u(y')}{2}\nonumber \\
                                                             &\le &\left( \frac{|x-x'|}{2} + \frac{|y-y'|}{2} -  \frac{1}{2\eta}\left\{ |x-x'|^2 + |y-y'|^2\right\}\right) \nonumber  \\
                                                             && + \epsilon u\left(\frac{x'+y'}{2}\right)-\frac{u(x') + u(y')}{2}. \nonumber
\end{eqnarray}
From our formal arguments above, it follows that 
$$
\lim_{|(x,y,x',y')|\rightarrow \infty}(w^\epsilon-\varphi_\eta)(x,y,x',y')=-\infty
$$
and, in particular, that there is $(x_\eta,y_\eta,x_\eta',y_\eta')\in \R^n\times \R^n\times \R^n\times \R^n$ maximizing $w^\epsilon-\varphi_\eta.$ By the Theorem of Sums (Theorem 3.2 in \cite{CIL}), for each $\rho > 0$ there are $X,Y\in {\mathcal S}(2n)$ such that 
{\small
$$
\left( D_x\varphi_\eta(x_\eta,y_\eta, x'_\eta, y'_\eta),D_y\varphi_\eta(x_\eta,y_\eta, x'_\eta, y'_\eta), X\right)\in \overline{J}^{2,+}\left((x,y)\mapsto \epsilon u\left(\frac{x+y}{2}\right)\right)\Big|_{x=x_\eta, y=y_\eta},
$$
$$
\left( -D_{x'}\varphi_\eta(x_\eta,y_\eta, x'_\eta, y'_\eta),-D_{y'}\varphi_\eta(x_\eta,y_\eta, x'_\eta, y'_\eta), Y\right)\in \overline{J}^{2,-}\left((x',y')\mapsto \frac{u(x')+u(y')}{2}\right)\Big|_{x'=x'_\eta, y'=y'_\eta},
$$
}
and 
\begin{equation}\label{matrixIneqConvex}
\left(
\begin{array}{cc}
X & 0\\
0 & -Y
\end{array}
\right)
\le A +\rho A^2.
\end{equation}
Here 
$$
A=D^2\varphi_\eta(x_\eta,y_\eta, x'_\eta, y'_\eta)=\frac{1}{\eta}
\left(
\begin{array}{cc}
I_{2n} & -I_{2n}\\
-I_{2n} & I_{2n}
\end{array}
\right).
$$
Note that the matrix inequality \eqref{matrixIneqConvex} implies $X\le Y$.

\par Set 
$$
p_\eta := D_x\varphi_\eta(x_\eta,y_\eta, x'_\eta, y'_\eta)=-D_{x'}\varphi_\eta(x_\eta,y_\eta, x'_\eta, y'_\eta)=\frac{x_\eta - x'_\eta}{\eta},
$$
$$
q_\eta := D_y\varphi_\eta(x_\eta,y_\eta, x'_\eta, y'_\eta)=-D_{y'}\varphi_\eta(x_\eta,y_\eta, x'_\eta, y'_\eta)=\frac{y_\eta - y'_\eta}{\eta},
$$
and also write
$$
X=\left(
\begin{array}{cc}
X_1 & X_2 \\
X_3 & X_4
\end{array}
\right)
\quad \text{and} \quad 
Y=\left(
\begin{array}{cc}
Y_1 & Y_2 \\
Y_3 & Y_4
\end{array}
\right)
$$
for appropriate $n\times n$ matrices $X_i, Y_i$ $i=1,\dots, 4$. As $X,Y\in {\mathcal S}(2n)$,  $X_1, X_4, Y_1, Y_4\in {\mathcal S}(n)$  and $X_2^t=X_3,\;  Y_2^t=Y_3$.

\par By direct verification, we have 
\begin{equation}\label{jetsconvex}
\begin{cases}
(p_\eta, X_1)\in J^{2,+}\left(x\mapsto \frac{\epsilon}{2}u\left(\frac{x+y_\eta}{2}\right)\right)\Big|_{x=x_\eta}\\
(q_\eta, X_4)\in J^{2,+} \left(y\mapsto \frac{\epsilon}{2}u\left(\frac{x_\eta +y}{2}\right)\right)\Big|_{y=y_\eta}\\
(p_\eta, Y_1)\in J^{2,-}\left(\frac{1}{2}u\right)(x'_\eta)\\
(q_\eta, Y_4)\in J^{2,-}\left(\frac{1}{2}u\right)(y'_\eta)
\end{cases}.
\end{equation}
Since the Lipschitz constant of the function $x\mapsto \epsilon u((x+y_\eta)/2)$ is less than or equal $\epsilon/2$, $|p_\eta|\le \epsilon/2<1/2$. Since
$p_\eta\in J^{1,-}\left(\frac{1}{2}u\right)(x'_\eta)$ and $u$ is a viscosity solution of \eqref{deltaPDE},

$$
\delta u(x'_\eta) - \tr Y_1 - f(x'_\eta)=0.
$$
Likewise, we conclude that 

$$
\delta u(y'_\eta) - \tr Y_4 - f(y'_\eta)=0.
$$

\par As $u$ is a viscosity solution of \eqref{deltaPDE}, we have from the first two inclusions in \eqref{jetsconvex}

$$
\delta u\left(\frac{x_\eta +y_\eta}{2}\right) -\frac{\tr X_1}{\epsilon} -  f\left(\frac{x_\eta +y_\eta}{2} \right)\le 0
$$
and

$$
\delta u\left(\frac{x_\eta +y_\eta}{2} \right) -\frac{\tr X_4}{\epsilon} -  f\left(\frac{x_\eta +y_\eta}{2} \right)\le 0.
$$
Averaging the two above inequalities gives

$$
\delta u\left(\frac{x_\eta +y_\eta}{2}\right) -\frac{\tr[X_1 +X_4]}{2\epsilon} -  f\left(\frac{x_\eta +y_\eta}{2}\right)\le 0.
$$
Altogether we have

\begin{eqnarray}\label{usingHconvex}
\delta \left\{\epsilon u\left(\frac{x_\eta +y_\eta}{2}\right) - \frac{u(x_\eta')+u(y_\eta')}{2}  \right\} &\le &\frac{1}{2}\tr[X_1- Y_1] + \frac{1}{2}\tr[X_4- Y_4] \nonumber \\
& &+  \epsilon f\left(\frac{x_\eta +y_\eta}{2}\right) - \frac{f(x_\eta')+f(y_\eta')}{2} \nonumber\\
&= & \frac{1}{2}\tr[X-Y] +  f\left(\frac{x_\eta +y_\eta}{2}\right) - \frac{f(x_\eta')+f(y_\eta')}{2} \nonumber \\
&\le &f\left(\frac{x_\eta +y_\eta}{2}\right) - \frac{f(x_\eta')+f(y_\eta')}{2}.
\end{eqnarray}

\par 4.  Another simple estimate for $w^\epsilon -\varphi_\eta$ is
\begin{eqnarray}
(w^\epsilon-\varphi_\eta)(x,y,x',y')&=& \epsilon u\left(\frac{x+y}{2}\right) - \frac{u(x)+u(y)}{2} +\nonumber \\
&& \frac{u(x)-u(x') + u(y) - u(y')}{2}- \frac{1}{2\eta}\left\{ |x-x'|^2 + |y-y'|^2\right\} \nonumber \\
&\le & \epsilon u\left(\frac{x+y}{2}\right) - \frac{u(x)+u(y)}{2} +\nonumber \\
&& \frac{|x-x'| + |y - y'|}{2}- \frac{1}{2\eta}\left\{ |x-x'|^2 + |y-y'|^2\right\} \nonumber \\
&\le &  \epsilon u\left(\frac{x+y}{2}\right) - \frac{u(x)+u(y)}{2} + \frac{\eta}{2}. \nonumber
\end{eqnarray}
This estimate implies that $(x_\eta, y_\eta)$ 
is a bounded sequence.  For otherwise, 
$(w^\epsilon-\varphi_\eta)(x_\eta,y_\eta,x_\eta', y_\eta')$ tends to $-\infty$ (by the above estimate on $w^\epsilon -\varphi_\eta $)
while 
\begin{eqnarray}
(w^\epsilon-\varphi_\eta)(x_\eta,y_\eta,x_\eta', y_\eta')&=&\max_{x,y,x',y'}(w^\epsilon-\varphi_\eta)(x,y,x', y') \nonumber \\
&\ge &(w^\epsilon-\varphi_\eta)(0,0,0, 0) \nonumber \\
& =& (\epsilon -1) u(0) \nonumber \\
&>& -\infty, \nonumber
\end{eqnarray}
for each $\eta>0.$  Likewise,  $(x'_\eta, y'_\eta)$ is a bounded sequence for all $\eta>0$ and small.  By Lemma 3.1 in \cite{CIL}, there is a cluster point $(x_\epsilon,y_\epsilon, x_\epsilon, y_\epsilon)$ of the sequence $((x_\eta, y_\eta,x_\eta',y_\eta'))_{\eta>0}$ through some sequence of $\eta\rightarrow 0$
that maximizes the function 
$$
(x,y)\mapsto \epsilon u\left(\frac{x+y}{2}\right) - \frac{u(x)+u(y)}{2}. 
$$
Passing to the limit through this sequence of $\eta$ tending to $0$ in \eqref{usingHconvex} gives for any $x,y\in\R^n$

$$
\epsilon u\left(\frac{x+y}{2}\right) - \frac{u(x)+u(y)}{2} \le f\left(\frac{x_\epsilon+y_\epsilon}{2}\right) - \frac{f(x_\epsilon)+f(y_\epsilon)}{2}\le 0
$$
due to the convexity of $f$. Sending $\epsilon\rightarrow 1^-$ establishes the claim. 
\end{proof} 
Aleksandrov's Theorem (section 6.4 of \cite{EG}) now implies the following corollary.

\begin{cor}
$u_\delta$ is twice differentiable at (Lebesgue) almost every point in $\R^n$.
\end{cor}
Since $u_\delta$ is convex and $f$ is superlinear, we expect
$$
\delta u_\delta(x) -\Delta u_\delta(x) -f(x)\le \delta u_\delta(x)-f(x)\le K + |x| - f(x)<0
$$
for all $x$ large enough and all $0<\delta\le 1$. Here $K$ is the constant in \eqref{viscsuper}. In other words, if $|Du_\delta(x)|<1$, then $|x|\le C$ for some $C$ independent 
of $0<\delta\le 1$. We give a precise statement of this in terms of jets.

\begin{cor}\label{technicalOmega}
There is a constant $C>0$, independent of $0<\delta\le 1$, such that if $|x|\ge C$ and $p\in J^{1,-}u_\delta(x)$, then $|p|=1$.
\end{cor}

\begin{proof}
Let $K$ be the constant in the \eqref{viscsuper} and choose $C$ so large that 
$$
K+|z|<f(z), \quad |z|\ge C.
$$
Recall that $J^{1,-}u_\delta(x)=\underline{\partial}u_\delta(x)$ by the convexity of $u_\delta$ (see Proposition 4.7 in \cite{BC}). Here 
$$
\underline{\partial}u(x)=\left\{p\in \R^n: u(y)\ge u(x) + p\cdot (y-x) \; \text{for all $y$ $\in \R^n$}\right\}
$$
is the {\it subdifferential} of $u$ at the point $x$. 

\par Moreover, $(p,0)\in J^{2,-}u_\delta(x)$,
and so 
$$
\max\{\delta u_\delta(x) - f(x), |p|-1\}\ge 0.
$$
As 
$$
\delta u_\delta(x)-f(x)\le K+|x|-f(x)<0,$$
$|p|=1$.
\end{proof}


\begin{lem}
Let $C_1>C$, where $C$ is the constant in the previous corollary. For almost every $x\in \R^n$, 

$$
D^2u_\delta(x)\le \frac{1}{\delta}\max_{|y|\le C_1}|D^2f(y)|
$$
for all $0<\delta<1$.
\end{lem}

\begin{proof}
1. Fix $0< \epsilon <1$, $0<|z|<C_1-C$, and set
$$
{\mathcal C}(x):=\epsilon u_\delta(x+z) - 2u_\delta(x) + \epsilon u_\delta(x-z), \quad x\in\R^n.
$$
As in previous arguments, we will give a formal proof (i.e. assuming $u\in C^2(\R^n)$)
first and then later describe how to our justify arguments.   For ease of notation, we write $u$ for $u_\delta$. 

\par As $\lim_{|x|\rightarrow \infty}u(x)/|x|=1$, 
$$
\lim_{|x|\rightarrow \infty}{\mathcal C}(x)=-\infty.
$$
Thus, there is $\hat{x}\in\R^n$ such that 
$$
{\mathcal C}(\hat{x})=\max_{x\in\R^n}{\mathcal C}(x).
$$
At $\hat{x}$, we have

$$
\begin{cases}
0=D{\mathcal C}(\hat{x})=\epsilon Du(\hat{x}+z) - Du(\hat{x})+ \epsilon Du(\hat{x}-z)\\
0\ge \Delta {\mathcal C}(\hat{x})=\epsilon \Delta u(\hat{x}+z) - \Delta u(\hat{x})+ \epsilon \Delta u(\hat{x}-z)
\end{cases}.
$$
Thus, 
$$
|Du(\hat{x})|=\frac{\epsilon}{2}|Du(\hat{x}+z)+Du(\hat{x}-z)|\le \epsilon <1
$$
and in particular
$$
\begin{cases}
|\hat{x}|\le C \; \text{(from the previous corollary)}\\
\delta u(\hat{x}) - \Delta u(\hat{x}) - f(\hat{x})=0
\end{cases}.
$$
Hence, for $x\in\R^n$
\begin{eqnarray}
\delta {\mathcal C}(x) & \le & \delta {\mathcal C}(\hat{x}) \nonumber \\
& = & \epsilon (\delta u(\hat{x} +z) + \delta u(\hat{x} - z)) - 2u(\hat{x}) \nonumber \\
&\le & \epsilon \Delta u(\hat{x}+z) - \Delta u(\hat{x})+ \epsilon \Delta u(\hat{x}-z) \nonumber \\
&& +  \epsilon (f(\hat{x} +z) + f(\hat{x} - z)) - 2f(\hat{x}) \nonumber \\
&\le &f(\hat{x} +z)  - 2f(\hat{x})+ f(\hat{x} - z)\nonumber \\
&\le & \max_{-1\le \xi \le 1}D^2f(\hat{x}+\xi z)z\cdot z\nonumber \\
&\le & \max_{|y|\le C_1}|D^2f(y)||z|^2\nonumber.
\end{eqnarray}
As the last expression is independent of $\epsilon$, we send $\epsilon\rightarrow 1^-$ and arrive at the inequality
$$
\frac{u(x+z)-2u(x) + u(x-z)}{|z|^2}\le \frac{1}{\delta}\max_{|y|\le C_1}|D^2f(y)|, \quad 0<|z|<C_1-C.
$$
The claim now follows as $D^2u$ exists a.e. in $\R^n$.

\par 2. Similar to previous proofs, we will ``triple the variables."  Again we fix $0<\epsilon<1$ and 
$0<|z|<C_1-C$. Set
$$
\begin{cases}
w(x_1, x_2, x_3):=\epsilon (u(x_1 +z) + u(x_2 - z) )- 2u(x_3)\\
\varphi_\eta(x_1, x_2, x_3):=\frac{1}{2\eta}\left\{|x_1-x_3|^2 + |x_2-x_3|^2  \right\}
\end{cases},
$$
for $x_1,x_2,x_3\in\R^n$ and $\eta>0$. Notice that

\begin{eqnarray}
(w-\varphi_\eta)(x_1, x_2, x_3)&=& \epsilon (u(x_1+z) - u(x_3-z) + u(x_2+z)-u(x_3-z) ) \nonumber\\
&& + {\mathcal C}(x_3) - \frac{1}{2\eta}\left\{|x_1-x_3|^2 + |x_2-x_3|^2\right\} \nonumber \\
&\le &\left( |x_1-x_3| - \frac{1}{2\eta}|x_1-x_3|^3\right) + \left( |x_2-x_3| - \frac{1}{2\eta}|x_2 - x_3|^3\right) +{\mathcal C}(x_3), \nonumber
\end{eqnarray}
which immediately implies 
$$
\lim_{|(x_1, x_2, x_3)|\rightarrow \infty}(w-\varphi_\eta)(x_1, x_2, x_3)=-\infty.
$$
In particular, there is $(x^\eta_1, x^\eta_2, x^\eta_3)$ globally maximizing $w-\varphi_\eta$.  Now we can invoke the Theorem of Sums and 
argue very similarly to how we did in the proof of the convexity of solutions of \eqref{deltaPDE}. We leave the details to the reader. 

\end{proof}

\begin{cor} We have the following:\\
(i) $u_\delta \in C^{1,1}(\R^n)$. \\
(ii)
$$
\Omega_\delta:=\left\{x\in\R^n: |Du_\delta(x)|<1 \right\}
$$
is open and bounded independently of all $0<\delta\le 1.$\\
(iii) $u_\delta\in C^{\I}(\Omega_\delta)$. \\
(iv) There is $L$ (independent of $0<\delta \le 1$) such that 
$$
D^2u_\delta(x)\le L,\quad x\in \Omega_\delta. 
$$
\end{cor}

\begin{proof}  As usual we write $u$ for $u_\delta $. $(i)$  is immediate from Proposition \ref{Basic2ndDer}. $(ii)$ follows from Corollary \ref{technicalOmega} and 
$(i)$, since $x\mapsto |Du(x)|$ is continuous on $\R^n$. $(iii)$ follows from basic elliptic regularity theory since
$u$ satisfies the linear elliptic PDE 
$$
\delta u(x) - \Delta u(x)=f(x), \quad x\in \Omega_\delta
$$
and $f\in C^\I(\R^n)$ (see Theorem 6.17 \cite{GT}).  As for $(iv)$, we have by convexity that if $x\in \Omega_\delta$ and $|\xi|=1$

\begin{eqnarray}
D^2u(x)\xi\cdot \xi &\le& \Delta u(x)\nonumber \\
&=& \delta u(x)-h(x) \nonumber \\
&\le & K + \delta |x|\nonumber \\
&\le & K + C=:L.\nonumber
\end{eqnarray}
\end{proof}


We conclude this subsection with a statement that $u_\delta$ is a Lipschitz extension of its values in $\overline{\Omega_\delta}$. 

\begin{prop}
\begin{equation}\label{LipExtUdelta}
u_\delta(x)=\min_{y\in \overline{\Omega_\delta}}\left\{u_\delta(y) + |x-y|\right\}, \quad x\in \R^n.
\end{equation}
\end{prop}
\begin{proof}
It is simple to check that, since Lip$[u_\delta]\le 1$, the formula above holds for $x\in\overline{\Omega_\delta}$. We now proceed to show that the formula above also holds
in the complement of $\overline{\Omega_\delta}$. 

\par As easy argument using the convexity of $u_\delta$ establishes that the minimum in \eqref{LipExtUdelta} is achieved on $\partial\overline{ \Omega_\delta}=\partial\Omega_\delta$
for $x\notin \Omega_\delta$. So we are left to show
\begin{equation}\label{LipExtUdelta2}
u_\delta(x)=\min_{y\in\partial\Omega_\delta}\left\{u_\delta(y) + |x-y|\right\}, \quad x\notin \Omega_\delta.
\end{equation}
To this end, we first notice that $u_\delta$ satisfies the eikonal equation

\begin{equation}\label{Eikonal}
\begin{cases}
|Dv(x)|=1, \; & x\in \overline{\Omega_\delta}^c\\
\hspace{.25in}v(x)=u_\delta(x), \; & x\in \partial{\Omega_\delta}
\end{cases}
\end{equation}
and we claim this PDE has a unique solution given by the right hand side of \eqref{LipExtUdelta2}.
It is not hard to see that the right hand side (RHS) above is a solution of \eqref{Eikonal}. RHS clearly defines a function with
Lipschitz constant at most 1 and hence is a subsolution of \eqref{Eikonal}. The RHS also dominates {\it every} subsolution of the eikonal equation that is
equal to $u_\delta$ on $\partial \Omega_\delta$ and therefore is a supersolution of \eqref{Eikonal} by Lemma 4.4 in \cite{CIL}.
The proof of uniqueness is a straightforward adaptation of the proof of comparison of sub- and supersolutions of \eqref{deltaPDE} (see also Theorem 5.9 of \cite{BC}). 
\end{proof}

\begin{cor}\label{D2uBoundedxLarge}
There is a universal constant $C>0$, such that the estimate 
$$
D^2u_\delta(x)\le \frac{1}{|x|-C}, \; \text{a.e.}\; |x|>C
$$
holds for all $0<\delta \le 1$.
\end{cor}

\begin{proof}
Recall that $\Omega_\delta$ is bounded independently of $0<\delta\le 1$; let $C$ be chosen so large that if $x\in \Omega_\delta$, then $|x|\le C$.
Also recall that $x\mapsto |x|$ is smooth on $\R^n\setminus\{0\}$ and that
$$
D^2|x|=\frac{1}{|x|}\left(I_n - \frac{x\otimes x}{|x|^2}\right)\le \frac{1}{|x|}I_n, \quad x\neq 0.
$$
Let $x\in \R^n$ with $|x|>C$. From \eqref{LipExtUdelta}, there is $y\in \partial \Omega_\delta$ such that $u_\delta(x)=u_\delta(y)+|x-y|$; moreover $|y|\le C.$ Also 
from \eqref{LipExtUdelta} and the above computation, we have that as $|z|\rightarrow 0$
\begin{eqnarray}
\frac{u_\delta(x+z) -2u_\delta(x)+u_\delta(x-z)}{|z|^2}&\le&\frac{|x+z-y| -2|x-y|+|x-z-y|}{|z|^2} \nonumber \\
&\le &\frac{1}{|x-y|}+o(1)\nonumber \\
&\le &\frac{1}{|x|-|y|}+o(1)\nonumber \\
&\le &\frac{1}{|x|-C}+o(1)\nonumber. 
\end{eqnarray}
The corollary now follows as $u_\delta$ is twice differentiable almost everywhere in $\R^n.$
\end{proof}

\subsection{$C^2$ regularity of solutions for rotational $f$}
We take a brief break from our analysis of $u_\delta$ for $\delta$ small to prove Theorem \ref{RotReg}, which asserts that when $f$ is rotational $u_\delta\in C^2(\R^n)$. We assume that
$$
f(x):=f_0(|x|), \quad x\in\R^n.
$$
where $f_0:[0,\infty)\rightarrow \R$ is convex, non-decreasing and superlinear.  Using the uniqueness of solutions of \eqref{deltaPDE}, it is straightforward to
show that $u_\delta$ will also be radially symmetric. That is,
$$
u_\delta(x)=\phi(|x|), \;x \in\R^n
$$
for some $\phi$ satisfying 

$$
\begin{cases}
\phi\in C^{1,1}(0,\infty)\\
\phi'\ge 0\\
\phi''\ge 0 \\
\lim_{r\rightarrow \infty}\frac{\phi(r)}{r}=1
\end{cases}
$$
and 
\begin{equation}\label{phiEq}
\max\left\{\delta \phi - \frac{n-1}{r}\phi' - \phi'' - f_0(r), \phi' -1\right\}=0, \quad r>0
\end{equation}
in the sense of viscosity solutions.  We claim that $\phi$ and therefore $u_\delta$ is necessarily $C^2$ (and not just $C^{1,1}$).

\par Since $\phi$ is non-decreasing, there is an $r_0>0$ such that
$$
\Omega_\delta=B_{r_0}(0).
$$
Therefore, it suffices to show that $\phi''(r)$ exists at $r=r_0$ and equals 0.  Of course, $\phi'(r)=1$, for $r\ge r_0$, so 
$$
\phi''(r_0+)=0
$$
and so we focus on establishing that 
$$
\rho:=\phi''(r_0-)=0.
$$

\par By the convexity of $\phi$, $\rho$ exists and is non-negative.  By \eqref{phiEq}, 

$$
\delta \phi - \frac{n-1}{r}\phi' - \phi'' - f_0(r)=0, \quad 0< r<r_0
$$
and so letting $r\rightarrow r_0^-$
\begin{equation}\label{phipp}
\delta \phi(r_0) - \frac{n-1}{r_0}\phi'(r_0) - \rho - h_0(r_0)=0.
\end{equation}
However, we always have
$$
\delta \phi - \frac{n-1}{r}\phi' - \phi'' - f_0(r)\le 0, \quad r>0
$$
and so letting $r\rightarrow r_0^+$ gives

\begin{equation}\label{phipp2}
\delta \phi(r_0) - \frac{n-1}{r}\phi'(r_0)  - f_0(r_0)\le 0.
\end{equation}
Comparing \eqref{phipp} and \eqref{phipp2}, we have 
$$
\rho=\delta \phi(r_0) - \frac{n-1}{r_0}\phi'(r_0) - f_0(r_0)\le 0.
$$
Thus, $\rho=0$ as desired which proves Theorem \ref{RotReg}.

\section{A uniform second derivative estimate}
According to Corollary \ref{D2uBoundedxLarge}, $D^2u_\delta$ is bounded from above  for all $x$ large enough independently of all $\delta$ positive and small. However, the upper bound we have in the whole space 
$$
\frac{1}{\delta}\max_{|y|\le C_1}|D^2f(y)|,
$$
blows up as $\delta\rightarrow 0^+.$  Our aim in this section is to obtain an estimate on the second derivative of $u_\delta$ that is {\it uniform} in all small $\delta>0$.  In fact, we prove
 
 \begin{prop}\label{Uniform2ndDer}
 For each ball $B\subset \R^n$, there is a constant $C=C(B)$ such that 
$$
 |Du_\delta(x)-Du_\delta(y)|\le C|x-y|, \quad x,y\in B
$$
for each $0<\delta\le1$. 
 \end{prop}
Having established the above proposition, we would immediately have from Corollary \ref{D2uBoundedxLarge} the following uniform second derivative estimate.  This bound, together with the convexity of
$u_\delta$, implies part $(ii)$ of Theorem \ref{deltathm}.

\begin{cor}
There is a universal constant $L$ such that 
\begin{equation}\label{FinalBound}
0\le D^2u_\delta(x)\le L, \; \text{a.e.} \; x\in\R^n
\end{equation}
for all $0<\delta\le 1$. 
\end{cor}

\subsection{Penalty method}

Towards proving Lemma \ref{Uniform2ndDer}, we fix $0< \delta \le 1$ and for $\epsilon$ positive and small study the solutions of the PDE

\begin{equation}\label{betaEq}
\begin{cases}
\delta v - \Delta v + \beta_\epsilon\left(|Dv|^2-1\right)=f, \quad x\in B\\
\hspace{1.68in} v=u_\delta, \quad x\in \partial B
\end{cases}.
\end{equation}
Here $(\beta_\epsilon)_{\epsilon>0}$ is family of functions satisfying

\begin{equation}\label{betaCOND}
\begin{cases}
\beta_\epsilon \in C^\infty(\R) \\
\beta_\epsilon = 0, \; z\le 0 \\
\beta_\epsilon >0, \; z>0\\ 
\beta_\epsilon'\ge 0\\
\beta_\epsilon''\ge 0 \\
\beta_\epsilon(z)=\frac{z-\epsilon}{\epsilon}, \quad z\ge 2\epsilon
\end{cases}.
\end{equation}
For each $\epsilon>0$, we think of $\beta_\epsilon$ as a smoothing of $z\mapsto (z/\epsilon)^+$. We consider this PDE a ``penalization" of equation \eqref{deltaPDE} as the values of $\beta_\epsilon(|Dv|^2-1)$ will
be large for small $\epsilon$, if $|Dv|^2\ge 1.$

As \eqref{betaEq} is a uniformly elliptic, semi-linear PDE it has a unique classical solution $v_{\epsilon}\in C^\I(B)\cap C^1(\overline{B})$ for each $\epsilon>0$ (by a straightforward variant of Theorem 15.10 in \cite{GT}).  Our goal is to deduce a pointwise
bound  $D^2v_\epsilon$ that is independent of all $\epsilon$ (and $\delta$) positive and small. With such an estimate we would be in a good position to pass to the limit and show $v_\epsilon\rightarrow u_\delta$ in  $C^1(B)$ and in particular that $u_\delta\in W^{2,\infty}(B)$. \footnote{We will actually do this on a slightly smaller ball than $B$. }

\par  This method was introduced by L.C. Evans to study elliptic equations with gradient constraints \cite{E}.  Although the results of \cite{E} apply to this problem,
 we derive our own estimates below because we need to understand how the estimates depend on $\delta$.  
 Our principle result related to this penalization scheme is

\begin{thm}\label{D2Vbounds}
Let $A\Subset B$. Then there is a constant $C=C(A)$ such that 
$$
|D^2v_\epsilon(x)|\le C, \; x\in A
$$
for all $0<\epsilon<1$. 
\end{thm}
We prove the above theorem by establishing several lemma.  We acknowledge that $v_\epsilon$ depends on $\delta$ and the overall goal is to obtain an estimate on $D^2v_\epsilon$ 
that is also independent of $\delta$.  For now, we consider $\delta>0$ to be fixed and later see how to establish such an uniform estimate while proving Theorem \ref{D2Vbounds} (see Corollary \ref{goodD2Vbounds}).

\begin{lem}
There is a constant $C$, independent of $\epsilon>0$, such that 
$$
|v_\epsilon(x)|\le C, \quad x\in \overline{B}.
$$
\end{lem}

\begin{proof}
\eqref{betaEq} enjoys a comparison principle for viscosity sub- and supersolutions.  It is easily verified
that $u_\delta$ is a viscosity sub-solution of \eqref{betaEq} and therefore
$$
u_\delta\le v_\epsilon.
$$
Likewise, the same comparison principle establishes
$$
v_\epsilon\le w
$$
where $w$ is the unique solution of the PDE

$$
\begin{cases}
\delta w - \Delta w =f, \; \; x\in B\\
\hspace{.55in}w=u_\delta, \; x\in \partial B
\end{cases}.
$$
\end{proof}

\begin{cor}
There is a constant $C$, independent of $\epsilon>0$, such that 
$$
|Dv_\epsilon(x)|\le C, \quad x\in \partial B.
$$
\end{cor}

\begin{proof} 
We have from the proof of the previous lemma that $u_\delta\le v_\epsilon\le w$ and equality holds on $\partial B$. Hence,
$$
\frac{\partial w(x)}{\partial \nu}\le \frac{\partial v_\epsilon(x)}{\partial \nu}\le \frac{\partial u_\delta(x)}{\partial \nu}, \quad x\in \partial B.
$$
\end{proof}


\begin{lem}
There is a constant $C$, independent of $0<\epsilon<1$, such that 
$$
|Dv_\epsilon(x)|\le C, \quad x\in \overline{B}.
$$
\end{lem}

\begin{proof}
For ease of notation we write $v$ for $v_\epsilon$. Set 
$$
\phi(x) := |Dv(x)|^2 - \lambda v(x),\; x\in \overline{B}
$$
for $\lambda>0$ which will be chosen below. To prove the claim, it suffices to bound $\phi$ from above.

\par  Direct computation gives

\begin{equation}\label{Dvineq}
\begin{cases}
D\phi=2D^2vDv-\lambda Dv\\
\Delta \phi\ge - (|Dv|^2 +C) + \beta'(|Dv|^2-1)\left(2Dv\cdot D\varphi + \lambda |Dv|^2\right)
\end{cases}
\end{equation}
for some constant $C$ independent of $\epsilon.$ Choose $x_0$ that maximizes $\phi.$ If $x_0\in \partial B$, the bound on $\phi$ follows from the previous lemma. If 
$$
\beta'(|Dv(x_0)|^2-1)< 1
$$
then 
$$
\beta'(|Dv(x_0)|^2-1)<\frac{1}{\epsilon}
$$
which implies $\beta(|Dv(x_0)|^2-1)\le 1$ and in particular \eqref{betaCOND} implies $|Dv(x_0)|^2-1\le 2\epsilon\le 2$. Therefore, if $x\in \partial B$ or $\beta'(|Dv(x_0)|^2-1)< 1$, then $\phi$ is bounded above. 

\par Alternatively, if  

$$
x_0\in B \quad \text{and} \quad \beta'(|Dv(x_0)|^2-1)\ge 1,
$$
we have from computation \eqref{Dvineq}

$$
0\ge - |Dv(x_0)|^2 +C +\lambda |Dv(x_0)|^2
$$
which, for $\lambda$ chosen large enough and independently of $0<\epsilon<1$ implies a bound on $|Dv(x_0)|^2$ and in turn on $\phi$. 
\end{proof}

\begin{lem}
For each $B'\Subset B$, there is a constant $C=C(B')$, independent of $0<\epsilon<1$, such that 
$$
\beta_\epsilon(|Dv_\epsilon(x)|^2-1)\le C, \quad x\in B'.
$$
\end{lem}

\begin{proof} 
1. It suffices to bound 
$$
\phi_\epsilon(x) = \xi(x)\beta_\epsilon(|Dv_\epsilon(x)|^2-1), \; x\in B
$$
for each $\xi\in C^\infty_c(B)$, $0\le \xi\le 1.$  For ease of notation we will omit the $\epsilon$ subscripts, function arguments and write $\beta$ for $\beta_\epsilon(|Dv_\epsilon|^2-1)$. Using the fact that $\beta$
is convex and that 
$$
\beta =\Delta v -\delta v + f\le C\{1+|D^2v|\}
$$
gives
\begin{eqnarray}
D\phi &=& \beta D\xi + \xi D\beta \nonumber \\
\Delta \phi &=& \Delta\xi \beta + 2  D\xi \cdot D\beta  + \xi \Delta\beta \nonumber \\
&=&  \Delta\xi \beta + 4\beta'  D\xi \cdot D^2vDv  + \xi (\beta''|2D^2vDv|^2 + 2\beta'(|D^2v|^2 + Dv\cdot D\Delta v) ) \nonumber \\
&\ge & - C(1+|D^2v|) - C\beta' |D^2v| + 2\beta' (\xi |D^2v|^2 + Dv\cdot \xi D\beta +\xi Dv\cdot D(\delta v -f))\nonumber \\
&\ge & - C(1+|D^2v|) + 2\beta'\left\{ \xi |D^2v|^2 + D\phi\cdot Dv  - \beta D\xi\cdot Dv-C|D^2v| - C\right\}\nonumber \\
&\ge & - C(1+|D^2v|) + 2\beta'\left\{ \xi |D^2v|^2 + D\phi\cdot Dv  -C|D^2v| - C\right\}\nonumber 
\end{eqnarray}
for various constants $C$ independent of $0<\epsilon <1$ (although dependent on $\xi$).

\par 3. Choose a maximizing point $x_0$ for $\phi$. If $x_0\in \partial B$,
$\phi\le \phi(x_0)=0.$ Now suppose that $x_0\in B$.
Necessarily 
$$
D\phi(x_0)=0\quad \text{and}\quad 0\ge \Delta \phi(x_0),
$$
and the above computations imply that at the point $x_0$
\begin{equation}\label{W2pIneq}
0\ge - C(1+|D^2v|) + 2\beta'\left\{ \xi |D^2v|^2 -C|D^2v| - C\right\}.
\end{equation}
If $\beta'\le 1<1/\epsilon$, then $\beta\le 1$ and thus $v$ is bounded uniformly from above. 
If $\beta'\ge 1$, \eqref{W2pIneq} gives
$$
0\ge\beta'\left\{ \xi |D^2v|^2 -C|D^2v| - C\right\}.
$$
As $\beta'>0$,
$$
0\ge \xi |D^2v|^2 -C|D^2v| - C
$$
which implies a bound on $\xi(x_0)|D^2v(x_0)|$ that is independent of $0<\epsilon<1.$ We conclude as
$$
\phi\le \phi(x_0)=\xi(x_0)\beta(|Dv(x_0)|^2-1)\le C\left(\xi(x_0)|D^2v(x_0)|+1\right)\le C.
$$
\end{proof}

\begin{proof}(of Theorem \ref{D2Vbounds})
1. It suffices to bound the quantity
$$
M_\epsilon:=\max_{x\in \overline{B}}\left\{\eta(x)|D^2v_\epsilon(x)|\right\}, 
$$
uniformly in all $\epsilon>0$ and small enough, for each $\eta\in C^\infty_c(B), 0\le \eta\le 1$.  To this end, we bound from the above the quantity
$$
\phi_\epsilon(x)=\frac{1}{2}\eta(x)^2|D^2v_{\epsilon}(x)|^2 + \eta(x)\lambda\beta_\epsilon(|Dv_\epsilon(x)|^2-1) +\frac{\mu}{2}|Dv_\epsilon(x)|^2, \quad x\in \overline{B}
$$
where $\lambda$ and $\mu$ are positive constants that will be chosen below. For ease of notation, we will omit $\epsilon$ dependence, function arguments and
write $\beta$ for  $\beta_\epsilon(|Dv_\epsilon|^2-1)$. We follow the arguments of Wiegner \cite{W} closely here. 

\par 2. As in previous arguments, we perform various computations that will help us study $\phi$ near its maximum value. 

\begin{align*}
\begin{cases}
\phi_{x_i}&=\eta\eta_{x_i}|D^2v|^2 + \eta^2D^2v\cdot D^2v_{x_i}+ \lambda(\eta_{x_i}\beta + 2\eta \beta' Dv\cdot Dv_{x_i}) + \mu Dv\cdot Dv_{x_i}, \quad i=1,\dots,n\\
\Delta \phi &= (\Delta\eta +|D\eta|^2)|D^2v|^2 + 4 \eta^2\sum^n_{i=1}\eta_{x_i}D^2v\cdot D^2v_{x_i} + \eta^2(|D^3v|^2 + D^2v\cdot D^2\Delta v)\\
& +\lambda \left[ \Delta\eta \beta + 4\beta'D^2vDv\cdot D\eta +\eta\left\{\beta''|2D^2vDv|^2 + 2\beta'(|D^2v|^2 + Dv\cdot D\Delta v)\right\}\right]  \\
& +\mu (|D^2v|^2 + Dv\cdot D\Delta v).
\end{cases}
\end{align*}
As 
$$
\Delta v=\beta +\delta v -f,
$$
we have for $i,j=1,\dots,n$
$$
\begin{cases}
(\Delta v)_{x_i}=2\beta'Dv\cdot Dv_{x_i} + \partial_{x_i}(\delta v -f)\\
(\Delta v)_{x_ix_j}=4\beta''(Dv\cdot Dv_{x_i})(Dv\cdot Dv_{x_j}) + 2\beta'(Dv_{x_i}\cdot Dv_{x_j} +Dv\cdot Dv_{x_ix_j})+\partial^2_{x_ix_j}(\delta v -f).
\end{cases}
$$
Substituting these values into the expression above we have for $\Delta v$ gives
\begin{align*}
\Delta \phi &= (\Delta\eta +|D\eta|^2)|D^2v|^2 + 4 \eta^2\sum^n_{i=1}\eta_{x_i}D^2v\cdot D^2v_{x_i} + \eta^2\text{\huge{[}}|D^3v|^2 +4\beta'' D^2v(D^2vDv)(D^2vDv) \\
& \left.+ 2\beta'\left(D^2v\cdot(D^2v)^2 + \sum^n_{i=1}v_{x_ix_j}Dv\cdot Dv_{x_ix_j}\right) + D^2v\cdot D^2(\delta v -f)\right]   +\lambda \left[ \Delta\eta \beta \right.\\
& \left. + 4\beta'D^2vDv\cdot D\eta+\eta\left\{\beta''|2D^2vDv|^2 + 2\beta'(|D^2v|^2 + 2\beta' D^2vDv\cdot Dv + Dv\cdot D(\delta v -f))\right\}\right]  \\
& + \mu (|D^2v|^2 + 2\beta'Dv\cdot D^2vDv + Dv\cdot D(\delta v-f)).
\end{align*}
Moreover, using our computation of $\phi_{x_i}$ gives our final expression for $\Delta\phi$
\begin{align}\label{D2vEstimate}
\Delta \phi &= (\Delta\eta +|D\eta|^2)|D^2v|^2 + 4 \eta\sum^n_{i=1}\eta_{x_i}D^2v\cdot \eta D^2v_{x_i} + \eta^2|D^3v|^2 + \lambda\Delta\eta\beta \nonumber \\
& 4\beta''\eta\{D^2v(D^2vDv)(D^2vDv) +\lambda |D^2vDv|^2\} +\eta^2D^2v\cdot D^2(\delta v -f) + 2\beta' Dv\cdot D\phi \nonumber \\
&+\beta'\left[\lambda\left(4 D^2vDv\cdot D\eta +2\eta|D^2v|^2 - 2\beta Dv\cdot D\eta + \eta Dv\cdot D(\delta v-f)\right) -2\eta |D^2v|^2 D\eta\cdot Dv\right]  \nonumber \\
& + \mu (|D^2v|^2  + Dv\cdot D(\delta v-f)).
\end{align}
\par 3. Choose a maximizing point $x_0$ for $\phi$. If $x_0\in \partial B$, we conclude as we already have a uniform gradient 
bound for $v$. Now suppose that $x_0\in B$. By calculus, 
$$
D\phi(x_0)=0\quad \text{and}\quad\Delta \phi(x_0)\le 0;
$$
and from equation \eqref{D2vEstimate}, we have at the point $x_0$
\begin{align}\label{D2vEstimate2}
0&\ge -C(|D^2v|^2+1) + 4\beta''|D^2vDv|^2\eta\{\lambda -\eta |D^2v|\} \nonumber \\
& \quad + 2\beta'\left[\lambda\left(\eta|D^2v|^2 - C - C|D^2v|\right) - C\eta|D^2v|^2\right] +\mu(|D^2v|^2-C)
\end{align}
for various constants $C$ independent of $\epsilon\in(0,1).$ In deriving the above inequality, we used the Cauchy-Schwarz inequality several times, the uniform bounds on
on $v$ and $|Dv|$, and also the fact that 
$$
\beta\le (|Dv|^2-1)\beta'\le C\beta'
$$
to simplify the term  $\lambda\Delta\eta\beta$. The inequality above is due to the uniform gradient bounds on $|Dv|$ and the inequality $\beta(z)\le z\beta'(z)$ which is immediate from 
the convexity of $\beta$ and $\beta(0)=0.$ 

\par 4. Now choose 
$$
\lambda=\lambda_\epsilon:=2M_\epsilon\ge 2\eta(x_0)|D^2v(x_0)|
$$
so that \eqref{D2vEstimate2} becomes
\begin{equation}\label{D2vEstimate3}
0\ge -C(|D^2v|^2+1) + 2\beta'\left[\lambda\left(\eta|D^2v|^2 - C - C|D^2v|\right) - C\eta|D^2v|^2\right] +\mu(|D^2v|^2-C).
\end{equation}
If for this choice of $\lambda$
$$
\lambda\left(\eta|D^2v|^2 - C - C|D^2v|\right) - C\eta|D^2v|^2\le 0,
$$
we would have a bound on $\eta(x_0)|D^2v(x_0)|$ independent of $\epsilon\in (0,1).$ If the above inequality does not hold, then from \eqref{D2vEstimate3} we infer
$$
 0\ge -C(|D^2v|^2+1) +\mu(|D^2v|^2-C).
$$
Clearly, there is a choice of $\mu>0$ so that $\eta(x_0)|D^2v(x_0)|$ is bounded independently of $\epsilon\in (0,1).$

\par 5. Finally, we observe that from our bounds on  $\eta(x_0)|D^2v(x_0)|$ 
\begin{align*}
M_\epsilon^2 &\le \max_{\overline{B}}\phi_\epsilon(x)\\
&= \frac{1}{2}\eta(x_0)^2|D^2v_{\epsilon}(x_0)|^2 + \eta(x_0)\lambda_\epsilon\beta_\epsilon(|Dv_\epsilon(x_0)|^2-1) +\frac{\mu}{2}|Dv_\epsilon(x_0)|^2\\
&\le\frac{1}{2}\eta(x_0)^2|D^2v_{\epsilon}(x_0)|^2+ CM_\epsilon +C \\
&\le C(M_\epsilon+1)
\end{align*}
for some constant $C$, independent of $\epsilon \in (0,1).$ It is now plain that $M_\epsilon$ is bounded independently of $\epsilon \in (0,1).$ 
\end{proof}
A close inspection of the above proof of Theorem \ref{D2Vbounds} reveals that we actually have the following crucial estimate. 
\begin{cor}\label{goodD2Vbounds}
Let $A\Subset B$. Then there is a constant $C'=C'(A)$ such that 
\begin{equation}\label{CRUCIALest}
|D^2v_\epsilon(x)|\le C'\left(1+|\delta v_\epsilon|_{L^\infty(B)} +|Dv_\epsilon|^2_{L^\infty(B)} \right), \; x\in A
\end{equation}
for all $0<\epsilon<1$. 
\end{cor}

The above estimates will now be used to establish Proposition \ref{Uniform2ndDer}. To this end, we first show here that there is a subsequence of $\epsilon\rightarrow 0^+$
such that $v_\epsilon\rightarrow u_\delta$ in $C^1_{\text{loc}}(B)$, where $u_\delta$ is the solution of \eqref{deltaPDE}. By the local, pointwise
estimates that we have established for $D^2v_\epsilon$, this convergence would imply $u_\delta\in W^{2,\infty}_{\text{loc}}(B).$
Then we use \eqref{CRUCIALest} to establish the inequality \eqref{FinalBound}.




\begin{proof} (of Proposition \eqref{Uniform2ndDer}) 1. So far we have established that there is a constant $C>0$ such that 
$$
|v_\epsilon|_{W^{1,\infty}(B)}\le C, \quad \epsilon \in (0,1),
$$
and for each $B'\subset B$, there is a constant $C'$ such that 
$$
|v_\epsilon|_{W^{2,\infty}(B')}\le C', \quad \epsilon \in (0,1).
$$
We claim that there is a function $v\in W^{1,\infty}(B)\cap W^{2,\infty}_{\text{loc}}(B)$ and a 
sequence of $\epsilon$ tending to 0 such that  as $\epsilon \rightarrow 0$
\begin{equation}
\begin{cases}
v_\epsilon\rightarrow v \quad \text{uniformly in $\overline{B}$}\\
v_\epsilon\rightarrow v \quad \text{in $C^1_{\text{loc}}(B)$}
\end{cases}.\footnote{That is, $v_\epsilon\rightarrow v$ uniformly in $\overline{B}$ and $v_\epsilon\rightarrow v$ in $C^1(B')$ for each $B'\Subset B$ through a sequence of $\epsilon\rightarrow 0.$}
\end{equation}

\par  Set
$$
B_{j}=\left\{x\in B: \text{dist}(x,\partial B)\ge \frac{1}{j}\right\}, \quad j\in \N.
$$
and observe that the sequence of compact sets $B_j$ is increasing and $B=\cup_{j\in \N}B_{j}$. Without loss of generality suppose $B_1\neq \emptyset$.  The above estimates and the Arzel\`{a}-Ascoli Theorem imply that there is a function $v_1\in W^{1,\infty}(B)\cap W^{2,\infty}(B_1)$ and a sequence $\epsilon^0_k\rightarrow 0$ as $k\rightarrow \infty$ such that $v_{\epsilon^0_k}\rightarrow v_1$ uniformly in $\overline{B}$ and $v_{\epsilon^0_k}\rightarrow v_1$ in $C^1(B_1)$ as $k\rightarrow \infty$.  

\par The uniform bounds we have on the $W^{2,\infty}(B_2)$ norm of the sequence $v_{\epsilon^0_k}$ implies again with the Arzel\`{a}-Ascoli Theorem that there is a function $v_2\in W^{1,\infty}(B)\cap W^{2,\infty}(B_2)$ and a sub-sequence $(\epsilon^1_k)_{k\ge 1}$ of $(\epsilon^0_k)_{k\ge 1}$ such that $v_{\epsilon^1_k}\rightarrow v_2$ uniformly in $\overline{B}$ and $v_{\epsilon^1_k}\rightarrow v_2$ in $C^1(B_2)$ as $k\rightarrow \infty$. By induction, we have for each $j\in \N$, there is a function $v_j\in W^{1,\infty}(B)\cap W^{2,\infty}(B_j)$ and a sub-sequence $(\epsilon^j_k)_{k\ge 1}$ of $(\epsilon^{j-1}_k)_{k\ge 1}$ such that $v_{\epsilon^{j}_k}\rightarrow v_j$ uniformly in $\overline{B}$ and $v_{\epsilon^j_k}\rightarrow v_j$ in $C^1(B_j)$ as $k\rightarrow \infty$.

\par   The diagonal sequence $(v_{\epsilon^k_k})_{k\in\N}$ is a subsequence of each $(v_{\epsilon^j_k})_{k\in\N}$ with $j$ fixed. Hence, this diagonal sequence converges uniformly on $\overline{B}$ to some $v\in W^{1,\infty}(\overline{B})$. Fix any $B'\Subset B$, and note that $B'\subset B_j$ for $j$ fixed and large enough. $(v_{\epsilon^k_k})_{k\in\N}$ being a subsequence of $(v_{\epsilon^j_k})_{k\in\N}$ converges in $C^1(B')\subset C^1(B_j)$ to $v$ as $k\rightarrow \infty.$

\par 2.  We now claim that $v$ is a viscosity solution of \eqref{deltaPDE} and therefore has to coincide with $u_\delta$, the unique viscosity solution of the PDE

$$
\begin{cases}
\max\left\{\delta v -\Delta v -f, |Dv|-1\right\}=0, \quad x\in B\\
\hspace{2in}v=u_\delta, \quad x\in \partial B
\end{cases}.
$$
Suppose that
$v-\varphi $ has a local maximum at $x_0\in B$ and that $\varphi\in C^2(B)$. We must show 

\begin{equation}\label{vviscsub}
\max\left\{\delta v(x_0)-\Delta\varphi(x_0) - f(x_0), |D\varphi(x_0)|-1 \right\}\le 0.
\end{equation}
By adding $x\mapsto \frac{\rho}{2} |x-x_0|^2$ to $\varphi$ and later sending $\rho\rightarrow 0$, we may assume that $v-\varphi $ has a {\it strict} local maximum.  Since, $v_{\epsilon_k}$ 
converges to $v$ uniformly (for some sequence $\epsilon_k\rightarrow 0$) as $k\rightarrow \infty$, there is a sequence of $x_k$ such that

$$
\begin{cases}
x_k\rightarrow x_0,\quad \text{as}\; k\rightarrow \infty\\
v_{\epsilon_k}-\varphi \;\;\text{has a local maximum at $x_k$}
\end{cases}.
$$
Since $v_\epsilon$ is a smooth solution of \eqref{betaEq}, we have 
$$
\delta v_{\epsilon_k}(x_k)- \delta\varphi(x_k) + \beta_\epsilon(|D\varphi(x_k)|^2-1)\le f(x_k).
$$
Since, $\beta_\epsilon \ge 0$, we can send $k\rightarrow \infty$ to arrive at
$$
\delta v(x_0)- \Delta\varphi(x_0)\le f(x_0). 
$$
By Theorem \eqref{D2Vbounds},
$$
0\le  \beta_\epsilon(|D\varphi(x_k)|^2-1)=\beta_\epsilon(|Dv_{\epsilon_k}(x_k)|^2-1)\le C.
$$
which implies that when $k\rightarrow \infty$
$$
|D\varphi(x_0)|^2-1\le 0.
$$
Thus, \eqref{vviscsub} holds. 

\par Now suppose that 
$v-\psi $ has a (strict) local minimum at $x_0\in B$ and that $\psi\in C^2(B)$. We must show 

\begin{equation}\label{vviscsup}
\max\left\{\delta v(x_0)- \Delta\psi(x_0) - f(x_0), |D\psi(x_0)|-1 \right\}\ge 0.
\end{equation}
Arguing as above, we discover there is a sequence $\epsilon_k\rightarrow 0$ as $k\rightarrow \infty$, and $x_k$ such that
$$
\begin{cases}
x_k\rightarrow x_0,\quad \text{as}\; k\rightarrow \infty\\
v_{\epsilon_k}-\psi \;\;\text{has a local minimum at $x_k$}
\end{cases}.
$$
If 
$$
|Dv(x_0)|^2\ge 1,
$$
then \eqref{vviscsup} holds.  Suppose now that 
$$
|Dv(x_0)|^2< 1.
$$
Since $v_\epsilon$ is a smooth solution of \eqref{betaEq}, we have 
\begin{equation}\label{LastStepPenalization}
\delta v_{\epsilon_k}(x_k)- \Delta\psi(x_k) + \beta_\epsilon(|D\psi(x_k)|^2-1)- f(x_k)\ge 0.
\end{equation}
By the convergence established in the first part of this proof, $|Dv_{\epsilon_k}(x_k)|<1$ for all $k$ sufficiently large. Hence,
$$
\lim_{k\rightarrow \I}\beta_\epsilon(|D\psi(x_k)|^2-1)=0. 
$$
With the above limit, we can pass to the limit \eqref{LastStepPenalization} to get
$$
\delta v(x_0)- \Delta\psi(x_0) - f(x_0)\ge 0
$$
and thus  \eqref{vviscsup} holds in this case, too. 
 
\par 3.  From \eqref{CRUCIALest}, we have that for $x,y\in B'\Subset B$, there is a constant $C'$ such that
$$
|Dv_{\epsilon_k}(x) - Dv_{\epsilon_k}(y)|\le C'\left(1+|\delta v_{\epsilon_k}|_{L^\infty(B)} +|Dv_{\epsilon_k}|^2_{L^\infty(B)} \right)|x-y|
$$
for all $k$ sufficiently large.  As $v_{\epsilon_k}\rightarrow u_\delta $ in $C^1_{ \text{loc} }(B)$ and as
$$
|\delta u_\delta|_{L^\infty(B)}+|Du_\delta|^2_{L^\infty(B)}\quad \text{is bounded for $0<\delta \le 1$},
$$
we let $k\rightarrow \infty$ to discover that there is a constant $L'=L'(B')$ such that 
$$
|Du_{\delta}(x) - Du_{\delta}(y)|\le L'|x-y|, \quad x,y\in B'.
$$
\end{proof}

\subsection{Passing to the limit}\label{pass2lim}
We now have the following estimates on $u_\delta$ $(0<\delta \le 1)$

$$
\begin{cases}
(|x| - K)^+\le u_\delta(x)\le \frac{K}{\delta} + |x|, \quad x\in \R^n\\
|Du_\delta(x)|\le 1, \quad x\in \R^n\\
|Du_\delta(x)-Du_\delta(y)|\le L|x-y|, \; x,y\in \R^n. 
\end{cases}
$$
Our aim is to pass to limit as $\delta\rightarrow 0^+$ and prove there is an eigenvalue $\lambda^*$ as stated in Theorem \ref{mainthm}. To this end, 
we define 
$$
\begin{cases}
\lambda_\delta:=\delta u_\delta(x_\delta)\\
v_\delta(x):=u_\delta(x)-u_\delta(x_\delta)\\
\end{cases}
$$
where $x_\delta$ is a global minimizer of $u_\delta$. Of course $Du_\delta(x_\delta)=0$, and in particular $x_\delta\in \Omega_\delta$. Moreover, 
Corollary \ref{technicalOmega} asserts that $|x_\delta|\le C$ for some constant $C$ independent of all $0<\delta \le 1$.

\par For this constant $C$, we have that
$$
0\le \lambda_\delta \le K+C
$$ 
and that $v_\delta $ satisfies
$$
\begin{cases}
|v_\delta(x)|\le |x|+C\\
|Dv_\delta(x)|\le 1\\
|Dv_\delta(x)-Dv_\delta(y)|\le L|x-y|
\end{cases},
$$
for all $x,y\in  \R^n$, $0<\delta<1$.  We will now use the above estimates to prove the following lemma. 


\begin{lem}\label{CONLEMMA}
There is a sequence $\delta_k>0$ tending to $0$ as $k\rightarrow \infty$, $\lambda^*\in \R$, and $u^*\in C^{1,1}(\R^n)$ such 
that 
\begin{equation}\label{vEst}
\begin{cases}
\lambda^*=\lim_{k\rightarrow \infty}\lambda_{\delta_k}\\
v_{\delta_k}\rightarrow u^* \;\text{in}\; C^1_{\text{loc}}(\R^n), \;\text{as}\; k\rightarrow \infty
\end{cases}.
\end{equation}
Moreover,  $u^*$ is a solution of \eqref{lamPDE} satisfying the growth condition \eqref{ugrowth} with eigenvalue $\lambda^*$ and 
\begin{equation}\label{ustarEst}
0\le D^2u^*\le L, \quad \text{a.e.}\; x\in \R^n.
\end{equation}
\end{lem}

\begin{proof}
Routine compactness and diagonalization arguments establishes the convergence \eqref{vEst}; similar arguments were used to prove Proposition \ref{Uniform2ndDer}. The estimate
\eqref{ustarEst} is also immediate from this convergence.

\par It is immediate from the definition that viscosity solutions pass to the limit under local uniform convergence. It follows that $u^*$ satisfies the PDE
$$
\max\{\lambda^* -\Delta u^* - h, |Du^*|-1\}=0, \quad x\in \R^n
$$
in the sense of viscosity solutions.  As $|u^*(x)|\le|x|+C$ for all $x\in \R^n$, 
$$
\limsup_{|x|\rightarrow \infty}\frac{u^*(x)}{|x|}\le 1.
$$
By the Lipschitz extension formula \eqref{LipExtUdelta} and Corollary \ref{technicalOmega}, we also have for all $|x|$ sufficiently large,
$$
v_{\delta}(x)=u_\delta(x)-u_\delta(x_\delta)\ge |x| -C
$$
for some $C$ independent of $0<\delta\le 1$. Thus,

$$
\liminf_{|x|\rightarrow \infty}\frac{u^*(x)}{|x|}\ge 1,
$$
and so $u^*$ satisfies the growth rate \eqref{ugrowth}.
\end{proof}
Note that for any $x_0\in \R^n$
$$
\delta u_\delta(x_0)=\delta u_\delta(x_\delta)+\delta (u_\delta(x_0)-u_\delta(x_\delta))=\delta u_\delta(x_\delta)+o(1)
$$
as $\delta\rightarrow 0^+.$ This follows since $|x_\delta|$ is bounded uniformly for all small $\delta>0$ and thus
$$
|\delta (u_\delta(x_0)-u_\delta(x_\delta))|\le \delta|x_\delta -x_0|\le C\delta, \quad 0<\delta\le1.
$$
Therefore, Lemma \ref{CONLEMMA} implies part $(iii)$ of Theorem \ref{deltathm}, and this fact with the comparison principle for eigenvalues (Proposition \ref{comparison}) proves
Theorem \ref{mainthm}.


\begin{rem}
We emphasize that main point to establishing a uniform second derivative estimate on $u_\delta$ was to show that $u_\delta-u_\delta(x_0)$ converges to 
to $u^*$ in $C^1_\text{loc}(\R^n)$ through some sequence of $\delta$ tending to $0$.  Uniform convergence would have followed without this estimate as
we have uniform Lipschitz estimates on $u_\delta.$
\end{rem}

\begin{rem}
As we established for $u_\delta$, $u^*$ is its own Lipschitz extension
$$
u^*(x)=\min_{y\in \overline{\Omega_0}}\left\{u^*(y) + |x-y|\right\}, \quad x\in \R^n
$$
where $\Omega_0=\{x\in \R^n: |Du^*(x)|<1\}$.  Therefore,  it suffices only to know $u^*$
within $\Omega_0$ to know it everywhere in space. 
\end{rem}

\section{Min-max formula}\label{MinMaxForm}
We conclude this work by proving Theorem \ref{MinMaxThm} which is an alternative, min-max characterization of the eigenvalue $\lambda^*.$   To this end, we recall formula \eqref{maxform}
\begin{eqnarray}
\lambda^*&=&\sup\text{{\huge\{} } \lambda \in\R: \text{there exists a subsolution $u$ of \eqref{lamPDE} with eigenvalue $\lambda$}, \nonumber \\
&& \hspace{2in}\left. \text{satisfying $\limsup_{|x|\rightarrow \infty}\frac{u(x)}{|x|}\le 1$.}  \right\}\nonumber
\end{eqnarray}
and formula \eqref{minform}
\begin{eqnarray}
\lambda^*&=&\inf\text{{\huge\{} } \mu \in\R: \text{there exists a supersolution $v$ of \eqref{lamPDE} with eigenvalue $\mu$}, \nonumber \\
&& \hspace{2in}\left. \text{satisfying $\liminf_{|x|\rightarrow \infty}\frac{v(x)}{|x|}\ge 1$.}  \right\},\nonumber
\end{eqnarray}
\noindent which are now consequences of Theorem \ref{mainthm} and the comparison principle established in Proposition \ref{comparison}.  Our goal is to use the above equalities to show
$$
\lambda_-=\lambda^*\le \lambda_+,
$$
where
$$
\lambda_-:=\sup\left\{\inf_{x\in\R^n}\left\{\Delta \phi(x) + f(x)\right\} :  \phi\in C^2(\R^n), |D\phi|\le 1 \right\}
$$ 
and
$$
\lambda_+:=\inf\left\{\sup_{|D\psi(x)|<1}\left\{\Delta \psi(x) + f(x)\right\} :  \psi\in C^2(\R^n), \liminf_{|x|\rightarrow \infty}\frac{\psi(x)}{|x|}\ge 1 \right\}.
$$
\begin{proof} (of Theorem \ref{MinMaxThm})  1. $(\lambda^*=\lambda_-)$ For $\phi\in C^2(\R^n)$ with $|D\phi|\le 1$, set
$$
\mu^\phi:=\inf_{x\in\R^n}\left\{\Delta \phi(x) + f(x)\right\}.
$$
If $\mu^\phi=-\infty$, then $\mu^\phi\le \lambda^*$.  If $\mu^\phi>-\infty$, then
$$
\max\{\mu^\phi - \Delta\phi(x)-f(x), |D\phi(x)|-1\}\le 0, \; x\in \R^n.
$$
Equality \eqref{maxform} implies $\mu^\phi\le \lambda^*$. Consequently, $\lambda_-=\sup\mu^\phi\le \lambda^*.$  

\par Now let $u^*$ be a convex, $C^{1,1}(\R^n)$ solution associated to $\lambda^*$ and $u^\epsilon:=\eta^\epsilon*u^*$ be the standard mollification 
of $u^*$ for $\epsilon>0$. Note that as $|Du^*|\le 1$ and $0\le D^2u^*\le L$, we have 
$$
|Du^\epsilon|\le 1 \quad \text{and} \quad 0\le D^2u^\epsilon \le L, \quad\text{for all} \quad \epsilon>0.
$$
Also note that as $u^*\in C^{1,1}$
$$
\Delta u^\epsilon =\eta^\epsilon*\Delta u^*\ge \lambda^* - f^\epsilon,
$$
where $f^\epsilon$ is the standard mollification of $f$.  

\par As $f$ grows superlinear and $D^2u^*$ is bounded, there is $R>0$ such that 
$x\mapsto \Delta u^\epsilon(x) + f(x)$ achieves its minimum value for an $x\in B_R$, for all $\epsilon>0$.   Hence, as $\epsilon\rightarrow 0^+$
 
 \begin{eqnarray}
 \lambda^* &\le & \inf_{|x|\le R}\left\{\Delta u^\epsilon(x) + f^\epsilon(x)\right\}\nonumber\\
& \le &\inf_{|x|\le R}\left\{\Delta u^\epsilon(x) + f(x)\right\} + o(1) \nonumber\\
& \le &\inf_{x\in \R^n}\left\{\Delta u^\epsilon(x) + f(x)\right\} + o(1) \nonumber\\
&\le & \lambda_- + o(1) \nonumber.
 \end{eqnarray}
 
 \par 2. $(\lambda^*=\lambda_+)$  Assume that $\psi \in C^2(\R^n)$ and that $\liminf_{|x|\rightarrow \infty}\psi(x)/|x|\ge 1.$  Similar to 
 our argument above, we set
 $$
 \tau^\psi:=\sup_{|D\psi(x)|< 1}\left\{\Delta \psi(x) + f(x)\right\}.
 $$
 If $\tau^\psi=+\infty$, then $\lambda^*\le  \tau^\psi.$ If $\tau^\psi<\infty$, then 
 $$
 \max\{ \tau^\psi- \Delta\psi(x) -f(x), |D\psi(x)|-1\}\ge 0, \; x\in \R^n.
 $$
 Equality \eqref{minform} implies $\lambda^*\le  \tau^\psi.$ Consequently, $\lambda^*\le \inf \tau^\psi=\lambda_+.$ This proves $(i)$.  
 
 \par If there is a $C^2(\R^n)$ supersolution $\psi^*$ of \eqref{lamPDE} with eigenvalue $\lambda^*$, such that $$\liminf_{|x|\rightarrow \infty}\frac{\psi^*(x)}{|x|}\ge 1,$$ then
 $$
 \lambda_+\le \sup_{|D\psi^*(x)|<1}\left\{\Delta \psi^*(x) + f(x)\right\}\le \lambda^*
 $$
 which proves assertion $(ii)$.
\end{proof}
We believe that the assumption on the existence of $\psi^*$ is not needed.  Our intuition is that the solution $u^*$ we constructed in Lemma \ref{CONLEMMA} is twice continuously 
differentiable on the set of points that $|Du^*|<1$, and therefore, should 
be amongst the class of $\psi$ in the infimum defining $\lambda_+$; in this case 
$$
\lambda^*= \sup_{|Du^*(x)|<1}\left\{\Delta u^*(x) + f(x)\right\}\ge \lambda_+.
$$
 
\begin{conj}
$\lambda^*=\lambda_+.$
\end{conj}

\noindent {\bf Acknowledgement}: I am indebted to Scott Armstrong for showing me the test function $\overline{u}$ defined in Lemma \ref{sarm}.

\appendix

\newpage


\end{document}